\documentclass[12pt, reqno]{amsart}

\usepackage{amsmath,amsthm,amssymb,mathdots,fullpage,enumerate,color}

\usepackage[all,cmtip]{xy}
\usepackage{tikz-cd}
\usepackage{xcolor}
\usepackage{bm}
\usepackage{upgreek}
\usepackage{mathtools}
\usepackage[active]{srcltx}
\usepackage{mathrsfs}
\usepackage{hyperref}
\definecolor{couleur_cite}{rgb}{0.05,.4,0.05}
\definecolor{couleur_link}{rgb}{0.05,0.05,0.4}
\hypersetup{
bookmarksopen,bookmarksnumbered,
colorlinks,
linkcolor=couleur_link,citecolor=couleur_cite,
}

\newtheorem{theorem}{Theorem}[section]
\newtheorem{lemma}[theorem]{Lemma}
\newtheorem{prop}[theorem]{Proposition}

\newtheorem{definition}[theorem]{Definition}

\theoremstyle{remark}
\newtheorem{remark}{Remark}

\newcommand{\R}{\mathbb R}
\newcommand{\C}{\mathbb C}
\newcommand{\N}{\mathbb N}
\newcommand{\Z}{\mathbb Z}
\newcommand{\Q}{\mathbb Q}

\newcommand{\A}{\mathbb A}

\newcommand{\ga}{\mathfrak a}

\newcommand{\qq}{\mathfrak q}
\newcommand{\dd}{\mathfrak d}
\newcommand{\NN}{\mathrm N}

\newcommand{\bK}{\mathbf K}

\newcommand{\SL}{\text{SL} }
\newcommand{\SO}{\text{SO} }

\newcommand{\GL}{\text{GL} }
\newcommand{\PSL}{\text{PSL} }

\newcommand{\tr}{\text{tr}}

\newcommand{\be}{\begin{equation}}
\newcommand{\ee}{\end{equation}}
\newcommand{\bes}{\begin{equation*}}
\newcommand{\ees}{\end{equation*}}

\renewcommand{\Re}{\mathrm{Re}}
\renewcommand{\Im}{\mathrm{Im}}
\renewcommand{\leq}{\leqslant}
\renewcommand{\geq}{\geqslant}
\renewcommand{\le}{\leqslant}

\numberwithin{equation}{section}

\makeatletter
\@namedef{subjclassname@2020}{%
  \textup{2020} Mathematics Subject Classification}
\makeatother

\title{Conductor zeta function for the $\GL(2)$ universal family}
\date{\today}

\renewcommand{\d}{\mathrm{d}}
\renewcommand{\mod}{\ \mathrm{mod} \ }

\allowdisplaybreaks

\author{Farrell Brumley}
\address{LAGA -- Institut Galil\'ee,
99 avenue Jean--Baptiste Cl\'ement,
93430 Villetaneuse,
France}
\email{brumley@math.univ-paris13.fr}

\author{Didier Lesesvre}
\address{School of mathematics (Zhuhai),
 Zhuhai Campus, Sun Yat-Sen University,
Tangjiawan, \\ \indent Zhuhai, Guangdong, 519082, China (PRC)}
\email{didier@mail.sysu.edu.cn, lesesvre@math.cnrs.fr}

\author{Djordje Mili\'cevi\'c}
\address{Bryn Mawr College, Department of Mathematics, 101 North Merion Avenue, Bryn Mawr, \\ \indent PA 19010, U.S.A.}
\email{dmilicevic@brynmawr.edu}

\begin{document}

\begin{abstract}
We obtain a Weyl law with power savings for the universal families of cuspidal automorphic representations, ordered by analytic conductor, of $\mathrm{GL}_2$ over $\Q$, as well as for Hecke characters over any number field. The method proceeds by establishing the requisite analytic properties of the underlying conductor zeta function.
\end{abstract}

\maketitle
\setcounter{tocdepth}{1}

\section{Introduction}
\label{sec:intro}

Weyl laws are central in the theory of automorphic forms in families \cite{sarnak_definition_2008}. The first instance of such a law in the non-compact setting was established, famously, by Selberg \cite{selberg_harmonic_1989} for Maass cusp forms of increasing eigenvalue for $\mathrm{SL}_2(\mathbb{Z})$. Since then, automorphic Weyl laws have been refined and extended for more general congruence quotients of noncompact locally symmetric spaces~\cite{lindenstrauss_existence_2007}. Of particular interest is to bound the error term in these asymptotics. An important milestone for subsequent applications is to  obtain a power savings error, and this has been recently achieved in the very general setting of reductive groups \cite{finis_remainder_2019}.

Relative to the entire cuspidal automorphic spectrum, these Weyl laws address \textit{partial} families, letting only the archimedean component of the representation vary. Moreover, this archimedean variation is often constrained to the spherical unitary dual. 
Recently the first and last authors \cite{brumley_counting_2018} obtained the first uniform Weyl law for the \textit{universal} family consisting of all cuspidal automorphic representations of $\mathrm{GL}_2$ (with partial results for $\GL_n$). The second author \cite{lesesvre_counting_2020} obtained the uniform Weyl law for unit groups of quaternion algebras. In both of these works, when the quaternion algebra is not totally definite, the techniques employed yielded error terms with logarithmic savings. We obtain in this paper, in a classical language, a Weyl law for the universal family of $\mathrm{GL}_2$ over~$\mathbb{Q}$ with a power savings error term. As a motivation and an illustration of the technique, we also prove the Weyl law for the universal family of Hecke characters on $\GL_1$ over a number field.

Before stating our precise results, we indicate an important methodological difference between the approach of this paper and that of \cite{brumley_counting_2018}  and \cite{lesesvre_counting_2020}. A long-standing analogy, which can be  traced to the work of Drinfeld \cite{Drinfeld}, relates the problem of counting automorphic forms of bounded conductor and the one of counting rational points of bounded height on algebraic varieties. A standard method in the latter field is to use the height zeta function, whose analytic properties yield the counting law through Tauberian arguments \cite{chambert-loir_igusa_2010}. This analogy has been a guide to our method and we obtain our result by studying, via the trace formula, the zeta function associated with the analytic conductor, demonstrating the efficiency of this new approach in the automorphic setting. Indeed, this technical device allows us to circumvent the arduous truncation procedure used in \cite{brumley_counting_2018} to obtain sharp cut-offs, which was the primary cause of the logarithmic savings. The price we pay for this is the use of test functions in the trace formula whose archimedean components are of non-compact support.

\subsection{Main result}
Fix a number field $F$ and an integer $n\geq 1$. Let $\mathbb{A}_F$ be the adele ring of $F$. Let $\mathfrak{F}_n$ be the family of isomorphism classes of unitary cuspidal automorphic representations of $\GL_n(\mathbb{A}_F)$ over $F$, whose central characters are normalized to be trivial on the diagonally embedded copy of the positive reals. The object of our study is the counting function 
\[
N_{\mathfrak{F}_n}(X)=|\{\pi\in\mathfrak{F}_n: c(\pi)\leq X\}|,
\]
where~$c(\pi)$ is the analytic conductor. We wish to obtain an asymptotic expression for~$N_{\mathfrak{F}_n}(X)$, with a power savings error term, in the two specific cases mentioned above ($n\leq 2$, under the assumption that $F=\Q$ when $n=2$). These two settings are particularly amenable to our analysis, due to the wide class of admissible test functions for the underlying trace formulae (Poisson summation and the Selberg trace formula). 

We shall in fact want to allow some flexibility in the choice of the archimedean component of the analytic conductor. Indeed our methods rely upon the nice analytic properties of certain related notions of archimedean conductor which the Iwaniec--Sarnak definition does not enjoy. This is an important aspect of the current work, one of whose aims is to investigate which definition of archimedean conductor is best suited for counting results of this type.

To set up the main theorem we introduce some notation. Write $r_1, r_2$ for the number of inequivalent real and complex embeddings of $F$ and set $r=r_1+r_2-1$. Let $\xi_F$ be the completed Dedekind zeta function of $F$, and write $\xi^*_F(1)$ for its residue at $s=1$. Let \[
{\rm vol}[\GL_n]=D_F^{n^2/2}\xi_F^*(1)\xi_F(2)\cdots \xi_F(n)
\]
denote the volume of the canonical measure of the adelic quotient $\GL_n(F)\backslash \GL_n(\A_F)^1$, where~$\GL_n(\A_F)^1$ is the quotient of $\GL_n(\A_F)$ by the diagonal copy of the positive reals. Let~$\widehat{\mu}_{\mathfrak{F}_n}$ be the positive finite measure on the adelic unitary dual of $\Pi(\GL_n(\A_F)^1)$ defined in~\cite[\S 1.4]{brumley_counting_2018}, whose volume can be computed to be
\[
{\rm vol}(\widehat{\mu}_{\mathfrak{F}_n})=\frac{\zeta_F^*(1)}{\zeta_F(n+1)^{n+1}}\int_{\GL_n(\R)^{1\wedge,{\rm temp}}}\frac{\d m^{\rm pl}_\infty(\pi_\infty)}{c(\pi_\infty)^{n+1}}.
\]
Finally, we put 
\[\mathscr{C}(\mathfrak{F}_n)= \frac{1}{n+1} {\rm vol}[\GL_n] {\rm vol}(\widehat{\mu}_{\mathfrak{F}_n}).\]

\smallskip

In this note, we prove the following result. 

\begin{theorem}
\label{thm}
Let $n=1$ and let $F$ be an arbitrary number field. Let $c(\pi)$ be the \emph{axiomatic analytic conductor}, in the sense of Definition \ref{defn:GL1-arch-cond}. Then
\begin{equation}\label{prop:GL1}
N_{\mathfrak{F}_1}(X)=\mathscr{C}(\mathfrak{F}_1) X^2 + O_\varepsilon\big( X^{2 - \frac{2}{r+2} + \varepsilon} \big).
\end{equation}
Let $n=2$ and $F=\Q$. Let $c(\pi)$ be an \emph{admissible analytic conductor}, in the sense of Definition \ref{defn-c}. Then
\begin{equation}\label{thm-eq:GL2}
N_{\mathfrak{F}_2}(X)= \mathscr{C}(\mathfrak{F}_2) X^3 + O_\varepsilon\big(X^{3-\frac23 + \varepsilon}\big).
\end{equation}
\end{theorem}

In a follow-up paper, we intend to extend this result to $\GL_2$ over an arbitrary number field by appealing to the Arthur-Selberg trace formula of Finis-Lapid \cite{finis_arthur-selberg_2011}.

\subsection{Comments on the proof}

As was already mentioned, we approach Theorem \ref{thm} through the use of a conductor zeta function. A similar approach for a universal Weyl law for tori was adopted in \cite{petrow_weyl_2018}. From a technical standpoint, this requires working with \textit{non-standard} test functions in the Selberg trace formula. Such test functions, having non-compact support at the archimedean place, allow us to avoid the smoothing procedure of~\cite{brumley_counting_2018}, which, by contrast, was executed at every fixed level and relied upon test functions of compact, but expanding (in terms of the level) support.

On the other hand, for our test functions to be \textit{admissible}, in the sense of Selberg, we must be careful in the definition we take of archimedean conductor  for Maass forms. We shall review the properties we need for the archimedean conductor in Section \ref{sec:notation-conductor}, but hasten to point out that we require that it continues analytically in the spectral parameter just beyond that of the trivial representation. Neither the original definition of Iwaniec-Sarnak~\cite{iwaniec_perspectives_2000}, the axiomatic definition of Booker \cite{booker_l-functions_2015}, nor the log-conductor of Conrey-Farmer-Keating-Rubinstein-Snaith \cite{conrey_integral_2005}, measuring the density of zeros near the critical point, satisfy this latter property. We hope that our methods shed some light into the nature of what should be considered as the ``correct'' definition of archimedean conductor.

Finally, we point out that, in contrast to \cite{brumley_counting_2018}, in which the error terms incurred by the trace formula estimates and a de-smoothing procedure were shown to be small enough to~(just) survive a sum on levels, in this paper we in fact exploit the sum on levels when we estimate the most difficult contribution of the trace formula, namely the hyperbolic terms (see Section~\ref{sec:hyperbolic-contribution}).

\subsection{Structure of the paper}

We begin by establishing the Weyl law for the family of Hecke characters over general fields in Section \ref{sec:GL1-case}. In this case the conductor zeta function and its study by means of trace formulae are already the key of the argument. Section \ref{sec:notation} is dedicated to the normalization of measures and the introduction of the different series of representations arising in $\mathrm{GL}_2(\R)$. In Section \ref{sec:conductor-limiting-measure} we introduce the axiomatic definition of the analytic conductor, and explain the limiting measure displayed in Theorem \ref{thm}. In Section~\ref{sec:holomorphic-forms} we establish the Weyl law for the discrete series, and in Section \ref{sec:maass-forms} we do so for the different principal series. 

Some appendices have been added to keep the paper self-contained and ease the reading. The specific form of the Selberg trace formula we use for the Maass form contribution is recalled in Appendix \ref{sec:STF}. The Weyl laws for Maass forms are deduced from a sharp Tauberian theorem that we state and prove in Appendix \ref{sec:tauberian}. A discussion and comparison with the alternative approach of using the theory of analytic newvectors, due to Jana and Nelson~\cite{jana_analytic_2019}, is given in Appendix~\ref{sec:alternative-approaches}.

\section{The $\GL_1$ case}
\label{sec:GL1-case}

The goal of this section is to prove the asymptotic count \eqref{prop:GL1}, while simultaneously introducing in this more elementary setting some of the principles of the proof of the corresponding result \eqref{thm-eq:GL2} for $\GL_2$ over $\Q$.

Let $F$ be a number field of degree $d$, let $\mathcal{O}_F$ its ring of integers and $\mathbb{A}_F$ its ring of adeles. Denote~$F_\infty=F\otimes_\Q\R$. Write $r_1, r_2$ for the number of inequivalent real and complex embeddings of $F$. Finally, set $r=r_1+r_2-1$.

\subsection{Archimedean characters}\label{sec:arch-char}
Recall that $\mathfrak{F}_1$ is the set of normalized Hecke characters of $F$. If $\chi=\otimes_v\chi_v\in \mathfrak{F}_1$ then $\chi_\infty=\otimes_{v\mid\infty}\chi_v$ is a continuous character of $F^\times_\infty/\R_+$, where $\R_+$ is the diagonally embedded copy of the positive reals. We may identify such characters with those of the norm 1 subgroup $F_\infty^1=\{x\in F_\infty^\times: |x|_\infty=1\}$. In this paragraph we shall recall a natural parametrization of $F_\infty^{1,\wedge}$, the group of continuous complex characters of $F_\infty^1$.

We may parametrize $F_\infty^{1,\wedge}$ as follows. Let $\mathcal{M}=\{0,1\}^{r_1}\times \Z^{r_2}$. For $m\in \mathcal{M}$, we define the unitary character $\delta_m\in F_\infty^{1,\wedge}$ by $\delta_m(x)=\prod_{v\mid\infty} (x_v/|x_v|_v)^{m_v}$, where $|\cdot |_v$ is the absolute value for $v=\R$ and the complex modulus for $v=\C$. Let 
\[
\ga_0=\bigg\{y\in \prod_{v|\infty}\R: \sum_{v\mid\infty} d_v y_v=0\bigg\}\qquad \text{where} \ d_v=[F_v:\R];
\]
then $\ga_0$ is of dimension $r$. A continuous complex character of $F_\infty^1$ can be written uniquely as
\[
\chi_{m, \nu}(x)=\delta_m(x)e^{\langle i\nu,\log |x|\rangle} \qquad (m\in\mathcal{M}, \;\nu\in \ga_{0,\C}),
\]
where $\ga_{0,\C}=\ga_0\otimes\C$. We call $\delta_m$ the \textit{discrete component} of $\chi_{m,\nu}$.

\subsection{Archimedean conductor}

We now define the conductor of the character $\chi_{m,\nu}\in F_\infty^{1,\wedge}$. In most settings, the archimedean conductor of Iwaniec--Sarnak, defined as
\[
c_{\mathrm{IS}}(\chi_{m,\nu})=\prod_{v\mid\infty} (1+|m_v+i\nu_v|)^{d_v},
\]
is sufficient. However, we shall need a notion of archimedean conductor that has better analytic properties as a function of $\nu\in\mathfrak{a}_{0, \C}$. For this reason, in the setting of $\GL_1$, we find it convenient to use the following definition, to be found in \cite[Definition 1.3]{booker_l-functions_2015}. 

\begin{definition}\label{defn:GL1-arch-cond}
The \emph{axiomatic archimedean conductor} associated with $\chi_{m,\nu}$ is defined by 
\[
c_{\mathrm{ax}} (\chi_{m,\nu})=\exp\sum_{v\mid\infty}2 \mathrm{Re}\, \frac{\Gamma'_v}{\Gamma_v}\left(\frac{1}{2}+\frac{|m_v|}{d_v}+i\nu_v\right), 
\]
where $\Gamma_\R(s) = \pi^{-s/2}\Gamma(s/2)$ and $\Gamma_\C(s) = (2\pi)^{-s} \Gamma(s)$.
\end{definition}

\begin{remark}\label{rem:GL1-cond-continuation}
As a function of $\nu\in \ga_0$, the axiomatic archimedean conductor extends to an analytic function in a rectangular tube about $\ga_0$ of side radius $1/2$ in $\ga_{0,\C}$. Indeed, the axiomatic conductor can be written as
\[
c_{\mathrm{ax}} (\chi_{m,\nu})=\exp \sum_{v\mid\infty}\left(\frac{\Gamma'_v}{\Gamma_v}\left(\frac{1}{2}+\frac{|m_v|}{d_v}+i\nu_v\right)+\frac{\Gamma'_v}{\Gamma_v}\left(\frac{1}{2}+\frac{|m_v|}{d_v}- i\overline{\nu_v}\right)\right).
\]
 This feature is of fundamental importance for us, in particular to obtain Lemma \ref{lem:BoundOnFourier} below, and motivates the definition of the analytic conductor for $\GL_2(\R)$ in Definition \ref{defn-c}.
\end{remark}

Let $m\in\mathcal{M}$ and $s\in \C$. Let $g_m$ be the function on $\ga_0$ defined by the Fourier transform
\begin{equation}
\label{DefinitionFourierTransform}
\begin{aligned}
g_m(x)=\int_{\ga_0}c(\chi_{m,\nu})^{-s}e^{- \langle i\nu,x\rangle}\,\d\nu.
\end{aligned}
\end{equation}
\pagebreak
\begin{lemma}\label{lem:BoundOnFourier} 
For any $\epsilon>0$ the integral in \eqref{DefinitionFourierTransform} converges uniformly and absolutely in the half-plane $\Re (s)\geq \frac{r}{r+1}+\epsilon$. Thus $g_m(x)$ defines in the half-plane $\Re (s)> \frac{r}{r+1}$ a holomorphic function of $s$. Moreover, letting $\|\cdot\|$ denote the euclidian norm on $\mathcal{M}$, there exists a $c_0>0$ such that for $\Re (s)> \frac{r}{r+1}$,
\[
g_m(x)\ll e^{-c_0\|x\|} (1+\|m\|)^{-\Re(s)}.
\]
\end{lemma}

\begin{proof}
According to \cite[Lemma 6.4]{brumley_counting_2018}, we have 
\[
\mathrm{vol}\{\nu\in\ga_0:c(\chi_{m,\nu})\leqslant X\}\ll X^{\frac{r}{r+1}}.
\]
Hence, the integral defining~$g(x)$ converges uniformly and absolutely in any half-plane of the form $\Re (s)\geqslant\frac{r}{r+1}+\delta$.

In view of Remark \ref{rem:GL1-cond-continuation}, for an arbitrary vector $c\in\R^d$ with $|c_v|<\frac12$ for all $v\mid\infty$, we may shift the contour in~\eqref{DefinitionFourierTransform} to $\ga_0-ic$, incurring no poles, and obtain
\[
g(x)=e^{-\langle c,x\rangle}\int_{\ga_0}c(\chi_{m,\nu-ic})^{-s}e^{-\langle i\nu,x\rangle}\,\text{d}\nu. 
\]
The problem of maximizing $\langle c,x\rangle$ subject to the conditions $|c_v|\leqslant\frac12-\delta$ has its solution in one of the finitely many vertices of a star-shaped polytope. The maximum is at least $c_0\|x\|$ for a fixed $c_0>0$.  Furthermore, by Stirling's formula we have the lower bound $c(\chi_{m, \nu}) \gg \log \|m\|$, so that $c(\chi_{m, \nu})^{-s}$ is bounded from above by $(1+\|m\|)^{-\Re(s)}$ as claimed.
\end{proof}

\subsection{Subfamilies}
Let $\chi=\otimes_v\chi_v\in \mathfrak{F}_1$. Let $\mathfrak{q}(\chi_f)\subset\mathcal{O}_F$ be the arithmetic conductor of $\chi_f$ and write $c(\chi_f)= \mathrm{N}\mathfrak{q}(\chi_f)\in\N$, where $\mathrm{N}$ is the norm on ideals of $F$. We define the \textit{analytic conductor} of $\chi$ by $c(\chi) = c(\chi_f) c(\chi_\infty)$, where $c(\chi_\infty)$ is the axiomatic analytic conductor. With this convention, one may then define the family $\mathfrak{F}_1(X)=\{\chi\in\mathfrak{F}_1: c(\chi)\leq X\}$, whose cardinality is counted by \eqref{prop:GL1}.

Our approach to proving \eqref{prop:GL1} will be to subdivide $\mathfrak{F}_1$ into continuous families, one for each discrete archimedean character $\delta_m$, as defined in \S\ref{sec:arch-char}. For $m\in\mathcal{M}$, we let
\[
\mathfrak{X}_m=\{\chi_{m,\nu}:\nu\in \ga_0\}
\]
denote the subset of unitary characters of $F_\infty^1$ with discrete component $\delta_m$. Clearly we have
\begin{equation}\label{eq:union-Xm}
F_\infty^{1,\wedge}=\coprod_{m\in\mathcal{M}}\mathfrak{X}_m.
\end{equation}
We give $\mathfrak{X}_m$ the Lebesgue measure $\d \nu$ induced by the natural identification of $\ga_0$ with $\R^{r-1}$. 

We decompose $\mathfrak{F}_1$ over $m\in \mathcal{M}$, by writing
\[
\mathfrak{F}_1=\coprod_{m\in\mathcal{M}}\mathfrak{F}_{\mathfrak{X}_m},
\]
where $\mathfrak{F}_{\mathfrak{X}_m}=\{\chi\in\mathfrak{F}_1: \chi_\infty\in\mathfrak{X}_m\}$. Accordingly, for every $m\in\mathcal{M}$ we put
\[
\mathfrak{F}_{\mathfrak{X}_m}(X)=\mathfrak{F}_{\mathfrak{X}_m}\cap\mathfrak{F}_1(X)\quad\textrm{ and }\quad N_m(X)=|\mathfrak{F}_{\mathfrak{X}_m}(X)|.
\]
From \eqref{eq:union-Xm} we have ${\rm vol}(\widehat{\mu}_{\mathfrak{F}_1})=\sum_{m\in\mathcal{M}}\widehat{\mu}_{\mathfrak{F}_1}(\mathfrak{X}_m)$, where $\widehat{\mu}_{\mathfrak{F}_1}(\mathfrak{X}_m)=\frac{\zeta_F^*(1)}{\zeta_F(2)^2}\int_{\mathfrak{X}_m}c(\chi_{m,\nu})^{-2}\d\nu$, as defined in Section \ref{sec:adelic-spectral-measure}.  The asymptotic count \eqref{prop:GL1} will follow from
\begin{equation}\label{eq:GL1-m-fixed-count}
N_m(X)=\frac12 {\rm vol} [\GL_1] \widehat{\mu}_{\mathfrak{F}_1}(\mathfrak{X}_m) X^2+O_\varepsilon( X^{2 - \frac{2}{r+2} + \varepsilon}).
\end{equation}

The remaining of this section is dedicated to prove this asymptotics.

\subsection{Conductor zeta function}
We shall show \eqref{eq:GL1-m-fixed-count} by establishing the nice analytic properties of the associated conductor zeta function
\begin{equation}
Z_m(s)=\sum_{\chi\in \mathfrak{F}_{\mathfrak{X}_m}}c(\chi)^{-s}.
\end{equation}
Indeed, the sharp asymptotic count \eqref{eq:GL1-m-fixed-count} follows from the following theorem, after applying the Tauberian Theorem \ref{thm:Tauberian}. 

\begin{prop}
The function $Z_m(s)$ satisfies the following properties:
\begin{enumerate}
\item it is holomorphic on $\sigma>2$;
\item it admits a meromorphic continuation to $\sigma>1$;
\item it has a unique pole in the right-half plane $\Re (s)>1$ located at $s=2$, which is simple and has residue ${\rm vol}[\mathrm{GL}_1] \widehat{\mu}_{\mathfrak{F}_1}(\mathfrak{X}_m)$;
\item it is of moderate growth in vertical strips. More precisely, uniformly in~$m \in~\mathcal{M}$, we have the bound $Z_m(s) \ll (1+\|m\|)^{-\sigma} (1+|s|)^{r/2}$.
\end{enumerate}
\end{prop}

Note that the above bound ensures that the sum of $Z_m(s)$ over $m \in \mathcal{M}$ converges uniformly in any half-plane $\sigma > 1+\delta$, so that the counting laws \eqref{eq:GL1-m-fixed-count} can be summed over $m$.

\begin{proof}
By splitting the sum according to non-archimedean conductor, we obtain
\begin{align*}
Z_m(s)&=\sum_{\qq}\NN\qq^{-s}\sum_{\substack{\chi\in \mathfrak{F}_{\mathfrak{X}_m}\\\qq(\chi_f)=\qq}}c(\chi_{\infty})^{-s} =\sum_{\qq}\NN\qq^{-s}\sum_{\dd\mid\qq}\mu\left( \frac{\qq}{\dd}\right)\sum_{\substack{\chi\in \mathfrak{F}_{\mathfrak{X}_m}\\ \qq(\chi_f)|\dd}}c(\chi_{\infty})^{-s}.
\end{align*}
We apply Poisson summation to the last sum. To set this up, we begin by defining a function~$h_m$ on $F^{1,\wedge}_\infty$ to take the value $c(\chi_{\infty})^{-s}$ if $\chi_\infty\in\mathfrak{X}_m$ and $0$ otherwise. Next, let~$U(\dd)$ be the open compact subgroup of $\A_{F,f}^\times$ which at each finite place is equal to the group of local units congruent to $1$ mod $\dd$ and let $\widehat{\varepsilon}_{U(\dd)}$ denote the Mellin transform of the normalized characteristic function $\varepsilon_{U(\dd)}={\rm vol}(U(\dd))^{-1}{\bf 1}_{U(\dd)}$. Then, assuming the measure on $\A_{F,f}^\times$ is normalized to give the local units volume 1, we obtain
\[
\widehat{\varepsilon}_{U(\dd)}(\chi_f)={\rm vol}(U(\dd))^{-1} \int_{U(\dd)}\chi_f^{-1}(u)\d^\times u=
\begin{cases}
1,& \qq(\chi_f)\mid\dd;\\ 
0,& \textrm{else}.
\end{cases}
\]
Write $H_{\dd,m}=\widehat{\varepsilon}_{U(\dd)}\otimes h_m$. With this notation, we may write 
\[
\sum_{\substack{\chi\in \mathfrak{F}_{\mathfrak{X}_m}\\ \qq(\chi_f)|\dd}}c(\chi_{\infty})^{-s}=\sum_{\chi\in \A_F^{1,\wedge}}H_{\dd,m}(\chi).
\]
Let $h_m^\vee$ be the inverse Mellin transform of $h_m$. Thus for $x\in F_\infty^1$ we have
\[
h_m^\vee(x)=\int_{F^{1,\wedge}_\infty}h_m(\chi_\infty)\chi_\infty^{-1}(x)\d\chi_\infty=\delta_m^{-1}(x)\int_{\ga_0}c(\chi_{m,\nu})^{-s}e^{-\langle i\nu,\log |x|\rangle}\d\nu.
\]
Then Poisson summation for the lattice $F^\times$ inside $\A^1_F$ states that
\[
\sum_{\chi\in \A_F^{1,\wedge}}H_{\dd,m}(\chi)={\rm vol}[\GL_1] \sum_{x\in F^\times} \left(\varepsilon_{U(\dd)}\otimes h_m^\vee\right)(x)={\rm vol}[\GL_1]\varphi(\dd)\sum_{u\in \mathcal{O}_F^\times(\qq)} h_m^\vee(u),
\]
where $\mathcal{O}_F^\times(\qq)=F^\times\cap U(\qq)=\{u\in\mathcal{O}_F^\times: u\equiv 1\mod\qq\}$ and $\varphi(\dd)=(\mu\star \mathrm{N})(\dd)={\rm vol}(U(\dd))^{-1}$. Thus we get
\begin{equation}
\label{ZF-transformed}
Z_m(s)={\rm vol}[\GL_1] \sum_{\qq}\NN\qq^{-s}\sum_{\dd\mid\qq}\mu\left(\frac{\qq}{\dd}\right) \varphi(\dd)\sum_{u\in\mathcal{O}_F^\times(\dd)}h_m^\vee(u).
\end{equation}

Splitting the sum \eqref{ZF-transformed}, we write $Z_m(s)=Z_m^1(s)+Z_m^{\neq 1}(s)$, where $Z_m^1(s)$ represents the contribution from $u=1$ in \eqref{ZF-transformed} and $Z_m^{\neq 1}(s)$ collects the contributions from $u\neq 1$.

For the $u=1$ summand, we apply Mellin inversion to $h_m^\vee(1)$, and we recognize the sum on $\dd\mid\qq$ as an arithmetic convolution $(\mu\star\mu\star\NN)(\qq)$. We therefore find
\begin{align*}
Z_m^1(s)&={\rm vol}[\GL_1] \sum_{\qq}\frac{(\mu\star\mu\star\NN)(\qq)}{\NN\qq^s}\ \int_{F_{\infty}^{1,\wedge}}h_m(\chi_{\infty})\,\text{d}\chi_{\infty}\\
&={\rm vol}[\GL_1] \frac{\zeta_F(s-1)}{\zeta_F(s)^2}\int_{\mathfrak{X}_m}c(\chi_{\infty})^{-s}\,\text{d}\chi_{\infty}.
\end{align*}
Finally, $Z_m^1(s)$ extends to a meromorphic function in the region $\Re (s)>1$ with the only pole at $s=2$, which is a simple pole with residue ${\rm vol}[\GL_1] \widehat{\mu}_{\mathfrak{F}_1}(\mathfrak{X}_m)$. From the vertical growth of the Riemann zeta function and the bound $c(\chi_{m, \nu}) \gg \|m\|$ we also deduce that, for $\sigma > 1$, we have  $Z_m^1(s) \ll (1+\|m\|)^{-\sigma} (1+|s|)^{r/2}$.

Note that $h_m^\vee(x)=\delta_m^{-1}(x)g_m(\log |x|)$ with $g_m$ defined by \eqref{DefinitionFourierTransform}. Using Lemma \ref{lem:BoundOnFourier}, we deduce that $Z_m^{\neq 1}(s)$ is dominated by
\begin{align*}
|Z_m^{\neq 1}(s)|&\leqslant {\rm vol}[\GL_1]  (1+\|m\|)^{-\sigma}\sum_{\qq}\NN\qq^{-\sigma}\sum_{\dd\mid\qq}\varphi(\dd)\sum_{\substack{u\in\mathcal{O}_F^{\times}\setminus\{1\}\\\varepsilon\equiv 1\bmod\dd}}e^{-c_0\|\log|u|\|}\\
& \leqslant {\rm vol}[\GL_1] \zeta_F(\sigma) (1+\|m\|)^{-\sigma} \sum_{u\in\mathcal{O}_F^{\times}\setminus\{1\}}e^{-c_0\|\log|u|\|}d_{1-\sigma}(u-1).
\end{align*}
uniformly in $m \in \mathcal{M}$. Writing $d_{1-\sigma}(u-1)\ll_{\varepsilon}\NN(u-1)^{\varepsilon}\ll_{\varepsilon}e^{\varepsilon\|\log|u|\|}$ to bound the divisor function and collecting everything, we find that $Z_m^{\neq 1}(s)$ is essentially bounded by $\zeta_F(\sigma)$ multiplied by $(1+\|m\|)^{-\sigma}$. In particular, the series defining~$Z_m^{\neq 1}(s)$ converges uniformly to a holomorphic function bounded by $(1+\|m\|)^{-\sigma}$ in any half-plane $\Re(s)\geqslant 1+\delta$.
\end{proof}

\section{$\GL_2$ preliminaries}
\label{sec:notation}

In this section we recall standard group decompositions of $\GL_2$, as a group over $\Q$, normalize Haar measures, describe the classification of local representations, and fix normalizations of local Plancherel measures. From now on we concentrate on the case of $\GL_2$ so we let~$\mathfrak{F} =  \mathfrak{F}_2$ and $\A=\A_\Q$ to lighten notations.

\subsection{Group decompositions}\label{sec:gp-decomp}
Let $G=\GL_2$, viewed as an algebraic group over $\Q$. Let~$B$ be the standard Borel subgroup of upper triangular matrices in $G$, and $N$ its unipotent radical. Let $T$ denote the diagonal torus. At finite places $p$ we write $\bK_p=\GL_2(\mathbb{Z}_{p})$; at the archimedean place we put $\bK_\infty={\rm O}(2)$ and $\bK_\infty^+=\SO(2)$. Then $\bK=\prod_v \bK_v$ is a maximal compact subgroup of $G(\A)$; let $\bK^+=\prod_p \bK_p\bK_\infty^+$. We have the local and global Iwasawa decompositions 
\begin{equation}\label{eq:local-Iwasawa}
G(\Q_p)=N(\Q_p)T(\Q_p)\bK_p,\qquad G(\R)=N(\R)T(\R)\bK_\infty=N(\R)T(\R)\bK_\infty^+,
\end{equation}
and
\[
G(\A)=N(\A)T(\A)\bK=N(\A)T(\A)\bK^+.
\]

Let $A_G=\{\left(\begin{smallmatrix}a & \\ & a\end{smallmatrix}\right): a\in \R_+\}$ denote the connected component of the split part of the center of $\GL_2(\R)$. Let $G(\A)^1=\{g\in G(\A): |\det g|_\A=1\}$ and $G(\R)^1=G(\A)^1\cap G(\R)$. Then $G(\R)^1$ (sometimes also denoted $\SL_2^\pm(\R)$) consists of the matrices $g\in\GL_2(\R)$ such that $\det g=\pm 1$. The natural matrix multiplication maps yield isomorphisms
\begin{equation}\label{G-G1}
G(\A)^1\times A_G\simeq G(\A),\quad G(\R)^1\times A_G\simeq G(\R),
\end{equation}
as well as
\begin{equation}\label{T-T1}
(T(\R)\cap G(\R)^1)\times A_G\simeq T(\R),\quad (T(\A)\cap G(\A)^1)\times A_G\simeq T(\A).
\end{equation}
Moreover, we have corresponding Iwasawa decompositions on $G(\R)^1$ and $G(\A)^1$ given by
\[
G(\R)^1=N(\R)(T(\R)\cap G(\R)^1)\bK_\infty^+\quad\textrm{ and }\quad G(\A)^1=N(\A)(T(\A)\cap G(\A)^1)\bK^+.
\]

\subsection{Haar measure normalizations}\label{sec:Haar}

For every place $v$ of $\Q$ we let $\d x_v$ be the Haar measure on $\Q_v$ given by the standard Lebesgue measure when $v=\R$ and the normalized Haar measure for which $\Z_p$ has volume $1$, when $v=p$ is finite. We identify $\Q_v$ with $N(\Q_v)$ by sending $x\in\Q_v$ to $\left(\begin{smallmatrix} 1 & x \\ & 1\end{smallmatrix}\right)$; the push forward of $\d x_v$ then defines a Haar measure $\d n_v$ on $N(\Q_v)$.

Similarly, we put a measure $\d^\times x_v$ on each $\Q_v^\times$, by taking $\d^\times x_\infty=\d x_\infty/|x_\infty|$ when $v=\infty$ and $\d^\times x_p=\zeta_p(1)\d x_p/|x_p|_p$ when $v=p$; note that in the latter case, $\Z_p^\times$ is given unit volume. This choice of measures induces one on $T(\Q_v)=\{\left(\begin{smallmatrix}a & \\ & b\end{smallmatrix}\right): a,b\in \Q_v^\times\}$, which we denote by~$\d t_v$, through its natural identification with $(\Q_v^\times)^2$.

Giving $\mathbf{K}_p$ and $\mathbf{K}_\infty^+$ the probability Haar measure, we then obtain a left-invariant Haar measure $\d g_v$ on $G(\Q_v)$ via the Iwasawa decomposition in \eqref{eq:local-Iwasawa}: $\d g_v=\d a_v \d n_v \d k_v$. Note that, for any prime $p$, this measure assigns $\mathbf{K}_p$ volume 1. This then defines a left-invariant Haar measure $\d g=\prod_v \d g_v$ on $G(\A)$.

We now endow $A_G$ with the Haar measure obtained by pushing forward $\d^\times t=\d t/t$ on $\R_+$ via $t\mapsto\left(\begin{smallmatrix} t & \\ & t\end{smallmatrix}\right)$. Since we have already chosen measures on $G(\R)$, $G(\A)$, $T(\R)$, and $T(\A)$ this induces measures on $G(\R)^1$, $G(\A)^1$, $(T(\R)\cap G(\R)^1)$, and $(T(\A)\cap G(\A)^1)$ via the topological group isomorphisms \eqref{G-G1} and \eqref{T-T1}.

Let ${\rm vol}[\GL_2]$ denote the volume assigned to the automorphic quotient $G(\Q)\backslash G(\A)^1$ by the above choice of measures. Then
\begin{equation*}\label{GL2-global-volume}
{\rm vol}[\GL_2]=\xi(2)=\frac{\pi}{6}.
\end{equation*}
The computation of this volume is classical, and can be deduced from \cite[Corollary 7.45]{knightly_traces_2006} after having taken into account their measure conventions.
\subsection{Normalization of Plancherel measure}
The choice of Haar measures in the preceding section fixes a normalization of Plancherel measure on the unitary dual, as we now recall.

For a place $v$ of $\Q$ we denote by $G(\Q_v)^\wedge$ for the unitary dual of $G(\Q_v)$, endowed with the Fell topology. We further let $G(\Q_v)^{\wedge,\textrm{temp}}$ and $G(\Q_v)^{\wedge,\textrm{gen}}$ be the tempered and generic unitary duals, endowed with the relative topologies. We have the following inclusions:
\[
G(\Q_v)^{\wedge,\textrm{temp}}\subset G(\Q_v)^{\wedge,\textrm{gen}}\subset G(\Q_v)^\wedge.
\]

For $\pi_v\in G(\Q_v)^\wedge$ and $f\in C_c^\infty(G(\Q_v))$ we put
\[
\pi_v(f)=\int_{G(\Q_v)}\pi_v(g_v) f(g_v)\d g_v,
\]
with Haar measure normalized as in \S\ref{sec:Haar}. This is a trace class operator on the space of $\pi_v$. We may define the distributional character of Harish-Chandra by taking the trace:
\[
\hat{f}(\pi_v)={\rm tr}\, \pi_v(f).
\]
In particular, for $v=p$ finite and $f=\varepsilon_p^\circ$ the characteristic function of $\bK_p$, we have $\widehat{\varepsilon_p^\circ}(\pi_p)=1$  for every unramified representation $\pi_p$ of $G(\Q_p)$.

The Plancherel measure $\d m_v^{\mathrm{pl}}$ on the dual group $G(\Q_v)^\wedge$ is supported on the tempered unitary dual~$G(\Q_v)^{\wedge, \textrm{temp}}$ and verifies the inversion formula
\[
f(e)=\int_{G(\Q_v)^{\wedge, \textrm{temp}}}\hat{f}(\pi_v)\d m_v^{\mathrm{pl}}(\pi_v)
\]
for all $f\in C^\infty_c(G(\Q_v))$, where $e\in G$ is the identity. 

\subsection{Archimedean spectrum and Plancherel measure}
We recall the explicit description of the tempered, generic, and unitary dual of $G(\R)^1$. For this, it will be useful to let~$\{\chi_+, \chi_-, \chi_1\}$ be the following characters of $T(\R)\cap G(\R)^1$:
\begin{enumerate}
\item[--] $\chi_+$ is the trivial character;
\item[--] $\chi_-$ is the character $\chi_- ({\rm diag}(a_1,a_2))= {\rm sgn}(a_1) {\rm sgn}(a_2)$;
\item[--] $\chi_1$ is the character $\chi_1 ({\rm diag}(a_1,a_2))= {\rm sgn}(a_1)$.
\end{enumerate}
For $\chi$ as above and $\nu\in\C$ we let $\pi_{\nu,\chi}$ denote the principal series representation unitarily induced from the character $\left(\begin{smallmatrix} a_1 & *\\ 0 & a_2\end{smallmatrix}\right)\mapsto \chi({\rm diag}(a_1,a_2))|a_1/a_2|^{i\nu}$ of the Borel subgroup~$B\cap~G(\R)^1$. In particular, the infinitesimal character of $\pi_{\nu,\chi}$ is $\frac14+\nu^2$. Then any $\pi\in G(\R)^{1\wedge}$ is isomorphic to one of the following types of representations:
\begin{enumerate}
\item the characters $\eta \circ \det$, where $\eta$ is either the trivial or sign character of $\{\pm 1\}$;
\item the discrete series representations, denoted $D_k$, for $k\geq 2$;
\item the weight zero even principal series representations $\pi_{\nu,\chi_+}$, where $\nu\in \R_{\geq 0}\cup i[0,1/2)$;
\item the weight zero odd principal series representations $\pi_{\nu,\chi_-}$, where $\nu\in \R_{\geq 0}\cup i(0,1/2)$;
\item the weight one principal series representations $\pi_{\nu,\chi_1}$, where $\nu\in \R_{>0}$;
\item the limit of discrete series representation, denoted $D_1$.
\end{enumerate}

In the above classification, the generic representations are the infinite dimensional ones; this eliminates the characters. Among the generic representations, the tempered ones are the discrete series representations and the various principal series representations with $\nu\in\R$. In particular, only the weight zero principal series representations can be non-tempered, namely, those with continuous parameter $\nu$ in the complementary segment $i(0,1/2)$.

Let $\mathcal{D}, \mathcal{P}_+, \mathcal{P}_-, \mathcal{P}_1^\circ$ denote the subsets of $\GL_2(\R)^{1\wedge, {\rm temp}}$ consisting of the discrete series representations, along with the tempered weight zero even, tempered weight zero odd, and weight 1 principal series representations of $G(\R)^1$, respectively.  Let $\mathcal{P}_1=\mathcal{P}_1^\circ\cup\{D_1\}$ denote the closure of the weight one principal series. We have
\[
\GL_2(\R)^{1\wedge, {\rm temp}}=\mathcal{D}\amalg \mathcal{P}_+\amalg \mathcal{P}_-\amalg \mathcal{P}_1.
\]

%
%

We now come to an explicit description of the Plancherel measure for $\GL_2(\R)^1$. This can be deduced\footnote{To cite a source such as \cite{lang_sl2R_1985} for the Plancherel measure of $\SL_2(\R)$, one must also compare normalizations of Haar measures. The restriction of the measure $\d g_\infty$ on $\GL_2(\R)^1$ defined in \S\ref{sec:Haar} to the open subgroup~$\SL_2(\R)$ yields a Haar measure which is $2$ times the Haar measure used in \textit{loc. cit}. To correct for this, one must multiply the Plancherel measure in \cite[p. 175]{lang_sl2R_1985} by a factor of $1/2$.}  from the Plancherel measure of the index two subgroup $\SL_2(\R)$, using the fact that the restriction of the principal series representations from $\GL_2(\R)^1$ to $\SL_2(\R)$ remain irreducible, whereas the restriction of $D_k$ to $\SL_2(\R)$ decomposes as a direct sum $D_k^+\oplus D_k^-$. One may calculate $\d m_\infty^{\mathrm{pl}}(\pi_\infty)$ as follows:
\begin{enumerate}
\item for the discrete series $D_k$ of weight $k\geq 2$, we have\footnote{For a direct computation of this, see \cite[Proposition 14.4]{knightly_traces_2006}, where the formal degree of $D_k$ is explicated with respect to the above choice of measures.}
\begin{equation}\label{eq:Dk-Plancherel}
\d m_\infty^{\mathrm{pl}}(D_k) = \frac{k-1}{4\pi};
\end{equation}
\item for the tempered weight zero even/odd principal series, we have
\[
\d m_\infty^{\mathrm{pl}}(\pi_{\nu,\chi_+})=\d m_\infty^{\mathrm{pl}}(\pi_{\nu,\chi_-})  = \frac{1}{2\pi} \nu \tanh(\pi\nu) \d \nu;
\]
\item for the (closure of the) weight one principal series, we have
\[
\d m_\infty^{\mathrm{pl}}(\pi_{\nu,\chi_1}) = \frac{1}{2\pi} \nu \coth(\pi\nu) \d \nu.
\]
\end{enumerate}
Therefore, the Plancherel inversion formula for $\GL_2(\R)^1$ reads
\[
f(e)=\frac{1}{2\pi}\int_0^\infty (h_++h_-)(\nu)\nu\tanh(\pi\nu) \d \nu+\frac{1}{2\pi}\int_0^\infty h_1(\nu)\nu\coth(\pi\nu) \d \nu+\frac{1}{4\pi}\sum_{k\geq 2}(k-1)h_k,
\]
where $h_k=\hat{f}(D_k)$ and $h_\bullet(\nu) = \hat{f}(\pi_{\nu, \chi_\bullet})$ for $\bullet \in \{+,-,1\}$.

\begin{remark}
In practice we shall treat the weight 0 even and odd principal series representations together. Letting $\mathcal{P}_0=\mathcal{P}_+\amalg \mathcal{P}_-$, we have the following decomposition
\begin{equation}\label{real-tempered-dual}
\GL_2(\R)^{1\wedge, {\rm temp}}=\mathcal{D}\amalg \mathcal{P}_0\amalg \mathcal{P}_1.
\end{equation}
Putting $h_0=h_++h_-$ and extending,  by parity, the integrals over the positive reals to all of~$\R$, we have
\[
f(e)=\frac{1}{4\pi}\int_\R h_0(\nu)\nu\tanh(\pi\nu) \d \nu+\frac{1}{4\pi}\int_\R h_1(\nu)\nu\coth(\pi\nu) \d \nu+\frac{1}{4\pi}\sum_{k\geq 2}(k-1)h_k.
\]
In particular, we may write
\begin{equation}\label{eq:P1-P2-Plancherel}
\begin{aligned}
\int_{\mathcal{P}_0}h(\pi_\infty)\d m^{\rm pl}_\infty(\pi_\infty)&=\frac{1}{4\pi}\int_\R h(\nu)\nu\tanh(\pi\nu) \d \nu,\\
\int_{\mathcal{P}_1}h(\pi_\infty)\d m^{\rm pl}_\infty(\pi_\infty)&=\frac{1}{4\pi}\int_\R h(\nu)\nu\coth(\pi\nu) \d \nu,
\end{aligned}
\end{equation}
\textit{for any function} $h$ for which the right-hand side converges.\end{remark}

\section{Conductors and the limiting measure $\hat{\mu}_\mathfrak{F}$}
\label{sec:conductor-limiting-measure}

In this section we discuss the local conductors and make sense of the measure $\hat{\mu}_\mathfrak{F}$ appearing in the statement of Theorem~\ref{thm}. 
We also introduce the subfamilies of $\mathfrak{F}$ corresponding to the different series of representations at the archimedean place, that will be addressed separately, thereby setting up the proof of Theorem \ref{thm}.

An automorphic representation $\pi \in \mathfrak{F}$ can be written as a restricted tensor product $\otimes_v \pi_v$ where $\pi_v \in G(\Q_v)^\wedge$, it will hence be sufficient to define the local analytic conductors $c(\pi_v)$ so that $c(\pi_v) = 1$ except for a finite number of places, and let $c(\pi) =\prod_v c(\pi_v)$.

\subsection{Local conductor at $p$}\label{sec:Casselman}
At a finite place $p$, the local conductor $c(\pi_p)$ of $\pi_p \in G(\Q_p)^{\wedge,{\rm gen}}$ is defined from the epsilon factor of its associated $L$-function. Casselman \cite{casselman_results_1973} proved that we have the equality $c(\pi_p) = p^{f({\pi_p})}$ where $f({\pi_p})$ is the smallest non-negative integer $f$ such that~$\pi_p$ admits nonzero fixed vectors by the congruence subgroup
\begin{equation}\label{eq:K1p-def}
K_{1,p}(f)=\left\{g=\begin{pmatrix} a & b\\ c & d\end{pmatrix}\in \bK_p: c\in p^f\Z_p,\; d-1\in p^f\Z_p\right\}.
\end{equation}
One deduces from this characterization, oldform dimension formulae, and the Plancherel inversion formula \cite[Section 6.1]{brumley_counting_2018} that
\begin{equation}\label{eq:local-meas-volume}
\int_{G(\Q_p)^{\wedge,\textrm{temp}}}\frac{\d m^{\rm pl}_p(\pi_p)}{c(\pi_p)^3}=\frac{\zeta_p(1)}{\zeta_p(3)^3},
\end{equation}
which will be of importance in \S\ref{sec:adelic-spectral-measure}. We note that the volume assigned to $K_{1,p}(f)$ by the measure $\d g_p$ in \S\ref{sec:Haar} is
\begin{equation}\label{eq:vol-K1(pf)}
{\rm vol}(K_{1,p}(f))=[\bK_p:K_{1,p}(f)]^{-1}=\frac{1}{\varphi_2(p^f)},
\end{equation}
where $\varphi_2=p_2 \star \mu$ and $p_2(d) = d^2$.

Note that $c(\pi_p) = 1$ for almost all places, so that the product over all the finite places of~$c(\pi_p)$ is well-defined. For later use, we introduce $K_1(q) = \prod_{p^f || q} K_{1,p}(f)$.

\subsection{Archimedean conductor}
\label{sec:notation-conductor}
The archimedean conductor was introduced in \cite{iwaniec_perspectives_2000} (with slightly different antecedents dating back at least to \cite{hoffstein_siegel_1995}). For example, when $\pi$ is the discrete series representation $D_k$, Iwaniec and Sarnak put $c_{\rm IS}(D_k)=k^2$. When $\pi_{\nu,\chi}$ is a unitary generic principal series representation, then they define the archimedean conductor~as
\begin{equation}\label{IS-cond}
c_{\textrm{IS}}(\pi_{\nu,\chi})=(1+|\nu|)^2.
\end{equation}
We shall adopt the same convention as Iwaniec--Sarnak for the discrete series representations, putting $c(D_k)=k^2$, but proceed slightly differently for the principal series representations. Indeed, we adopt the following axiomatic definition.

\begin{definition}\label{defn-c}
Let $\delta>0$. Let ${\rm Cond}_\delta(\C)$ denote the class of complex-valued functions $c$ on~$\C$ satisfying the following properties:
\begin{enumerate}
\item\label{cond1} even;
\item\label{cond2} $c(\nu)$ is holomorphic on the strip $|\Im(\nu) | \leqslant \tfrac12 + \delta$;
\item\label{cond3} $c(\nu) \asymp c_{\mathrm{IS}}(\pi_{\nu,\chi})$ for $\nu\in\R$;
\item\label{cond4} $c(\nu)$ is non-zero on the strip $|\Im(\nu) | \leqslant \tfrac12+\delta$;
\item\label{cond5} $c(\nu)$ is real and non-negative on $\R\cup i[\tfrac12, \tfrac12 ]$.
\end{enumerate}
Fix $c\in {\rm Cond}_\delta(\C)$. Let $\pi_{\nu,\chi}\in G(\R)^{1\wedge}$ be a unitary generic principal series representation. We call the complex number $c(\nu)$ the \emph{conductor} of $\pi_{\nu,\chi}$. We shall sometimes write it~$c(\pi_{\nu,\chi})$.
\end{definition}

Henceforth we shall fix a choice of $c\in {\rm Cond}_\delta(\C)$. An example of an eligible function is given by $c(\nu)=1+\nu^2\in {\rm Cond}_\delta(\C)$, for any $0<\delta<1/2$, which is a slight variant of \eqref{IS-cond}.

\begin{remark}\label{rem:admissibility}
Conditions \eqref{cond1}-\eqref{cond3} assure that, for ${\rm Re}(s) > 1+\delta$, the function $h_s(\nu)=c(\nu)^{-s}$ is admissible in the Selberg trace formula (we recall these conditions in Definition \ref{defn-admissible}). Note that such a function is not in the Paley--Wiener class since, by \eqref{cond3}, it is not of rapid decay in horizonal strips.

While Condition \eqref{cond3} ensures the asymptotic compatibility with the definition of Iwaniec--Sarnak, the analyticity requirement \eqref{cond2} does not in fact apply to their definition.
\end{remark}

\begin{remark}
Condition \eqref{cond2} just barely fails for the axiomatic conductor of \cite{booker_l-functions_2015} and the log-conductor of \cite{conrey_integral_2005}, which are holomorphic in any strip of the form $|\Im(\nu) | \leqslant \tfrac12 - \varepsilon$.
\end{remark}

\begin{remark}\label{rem:arch-int}
We have been unable to interpret the archimedean integral
\begin{equation}\label{eq:arch-int}
\int_{G(\R)^{1\wedge, \mathrm{temp}}} \frac{\d m^{\rm pl}_\infty(\pi_\infty)}{c(\pi_\infty)^{3}}
\end{equation}
appearing in the volume constant  \eqref{MT-meas-vol} of the main term, as we do for the analogous integral at the finite places in \eqref{eq:local-meas-volume}. In the latter case, the numerical identity \eqref{eq:local-meas-volume} results from the evaluation at $s=3$ of the local conductor zeta function
\[
\int_{G(\Q_p)^{\wedge,\textrm{temp}}}\frac{\d m^{\rm pl}_p(\pi_p)}{c(\pi_p)^{s}},
\]
which we succeed in identifying, as a function of $s$, with the zeta quotient $\zeta_p(s-2)/\zeta_p(s)$ in the half-plane ${\rm Re}(s)>2$; see \cite[Section 6.1]{brumley_counting_2018}. A similar interpretation of \eqref{eq:arch-int} would be sensitive to the exact definition one uses for $c(\pi_\infty)$. One natural candidate would be to replace $c(\pi_\infty)^{-s}$ with the gamma factor $\gamma(s+1/2,\pi_\infty)$, in view of $c(\pi_\infty\otimes |\det|^{-it})^{-\sigma}\asymp~\gamma(1/2+s,\pi_\infty)$, where we write $s=\sigma+it$. Thus one would consider
\[
\int_{G(\R)^{1\wedge, \mathrm{temp}}}\gamma(1/2+s,\pi_\infty)\d m^{\rm pl}_\infty(\pi_\infty);
\]
it seems plausible that this integral can be evaluated explicitly as a function of $s$. The use of $\gamma(1/2+s,\pi_\infty)$ would have the advantage of yielding a canonically defined integral, but~$\gamma(1/2,\pi_\infty)$ does not (quite) satisfy condition \eqref{cond2} of Definition \ref{defn-c}; see \cite[Prop. 9.5]{C-PS}.
\end{remark}

\subsection{The measure $\hat\mu_\mathfrak{F}$}
\label{sec:adelic-spectral-measure}

We now let $\Pi(G(\A)^1)^{\rm temp}$, $\Pi(G(\A)^1)^{\rm gen}$, and $\Pi(G(\A)^1)$ denote the direct product of the local tempered, generic, and  unitary duals, respectively, endowed with the product topology.\footnote{We emphasize that a $\pi\in\Pi(G(\A)^1)$ is not necessarily admissible.} 
We define $m^{\rm pl}_\A=\prod_v m^{\rm pl}_v$ to be the product of the Plancherel measures on $\Pi(G(\A)^1)^{\rm temp}=\prod_p G(\Q_p)^{\wedge, {\rm temp}}\times G(\R)^{1\wedge, {\rm temp}}$. Note that $m^{\rm pl}_\A$ (indeed each $m^{\rm pl}_v$) is an infinite measure. 

We define a positive \textit{finite} measure $\hat\mu_\mathfrak{F}$ on $\Pi(G(\A)^1)^{\rm temp}$ as follows. Let $\Omega=\prod_v\Omega_v$ be a basic open set, where for $v<\infty$ (resp. $v=\infty$) $\Omega_v$ is open in $G(\Q_v)^{\wedge,\rm{temp}}$ (resp. $G(\R)^{1 \wedge, \rm {temp}}$) and $\Omega_p=G(\Q_p)^{\wedge,\rm{temp}}$ for almost every $p$. Define
\[
\hat\mu_\mathfrak{F}(\Omega)=\int_\Omega^\star  \frac{\d m^{\rm pl}_\A(\pi)}{c(\pi)^3},
\]
where the regularized integral is defined by
\[
\int_{\Omega_\infty}\frac{\d m^{\rm pl}_\infty(\pi_\infty)}{c(\pi_\infty)^3} \prod_p \zeta_p(1)^{-1} \int_{\Omega_p}\frac{\d m^{\rm pl}_p(\pi_p)}{c(\pi_p)^3}.
\]
In particular, taking $\Omega=G(\A)^1$ and using \eqref{eq:local-meas-volume}, the total volume of $\hat\mu_\mathfrak{F}$ can be evaluated as
\begin{equation}\label{MT-meas-vol}
{\rm vol}(\hat\mu_\mathfrak{F}) = \int_{\Pi(G(\mathbb{A})^1)^{\rm temp}}^\star \frac{\d m^{\rm pl}_\A(\pi)}{c(\pi)^3}=\frac{1}{\zeta(3)^3}\int_{G(\R)^{1\wedge}}\frac{\d m^{\rm pl}_\infty(\pi_\infty)}{c(\pi_\infty)^3}.
\end{equation}

\subsection{Subfamilies and their counting functions}
For $\bullet\in\{\mathcal{D},\mathcal{P}_0,\mathcal{P}_1\}$, we let
\[
\Omega_\bullet=\Pi(G(\A_{f}))^{\rm temp}\times \bullet\subset \Pi(G(\A)^1)^{\rm temp}.
\]
Corresponding to \eqref{real-tempered-dual} we have $\Pi(G(\A)^1)^{\rm temp}=\Omega_{\mathcal{D}}\amalg\Omega_{\mathcal{P}_0}\amalg\Omega_{{\mathcal{P}_1}}$. Inserting \eqref{MT-meas-vol} we find
\begin{align*}
{\rm vol}(\hat\mu_\mathfrak{F})&={\rm vol}_{\hat\mu_\mathfrak{F}}(\Omega_\mathcal{D})+{\rm vol}_{\hat\mu_\mathfrak{F}}(\Omega_{\mathcal{P}_0})+{\rm vol}_{\hat\mu_\mathfrak{F}}(\Omega_{\mathcal{P}_1})\\
&=\frac{1}{\zeta(3)^3}\left(\int_\mathcal{D}\frac{\d m^{\rm pl}_\infty(\pi_\infty)}{c(\pi_\infty)^3}+\int_{\mathcal{P}_0}\frac{\d m^{\rm pl}_\infty(\pi_\infty)}{c(\pi_\infty)^3}+\int_{\mathcal{P}_1}\frac{\d m^{\rm pl}_\infty(\pi_\infty)}{c(\pi_\infty)^3}\right).
\end{align*}
Similarly, we let $\mathfrak{F}_\bullet$ denote those $\pi\in\mathfrak{F}$ for which $\pi_\infty\in\bullet$. Let
\[
N_\bullet(X)=|\{\pi\in\mathfrak{F}_\bullet: c(\pi)\leq X\}|,
\]
so that $N_\mathfrak{F}(X)=N_\mathcal{D}(X)+N_{\mathcal{P}_0}(X)+N_{\mathcal{P}_1}(X)$. Note that $N_{\mathcal{P}_0}(X)$ counts the weight zero even and odd Maass forms simultaneously. We have combined their contributions since the classical trace formulae we use do not distinguish them.

We shall prove Theorem \ref{thm} by showing that 
\begin{align*}
N_\mathcal{D}(X)&=\frac{1}{3} \xi(2) {\rm vol}_{\hat\mu_\mathfrak{F}}(\Omega_\mathcal{D}) X^3 + O(X^2),\\
N_{\mathcal{P}_k}(X)&=\frac{1}{3} \xi(2) {\rm vol}_{\hat\mu_\mathfrak{F}}(\Omega_{\mathcal{P}_k}) X^3 + O_\varepsilon(X^{3-\frac23 + \varepsilon})\qquad (k=0,1).
\end{align*}
These will be established in Propositions \ref{prop:counting-law-discrete} and \ref{prop:counting-law-Maass} below.

\subsection{From adelic to classical}
For the convenience of the reader, we provide a dictionary from the adelic setting of this section (as well as the introduction, where our main theorem was stated) to the classical setting (in which we prove the theorem). This will in particular be of use when we quote the Selberg trace formula in Appendix \ref{sec:STF}.

As usual, we let $[\GL_2]=\GL_2(\Q)\backslash\GL_2(\A)^1$ denote the automorphic space for $\GL_2$ over $\Q$. Recall the open compact subgroup $K_1(q)\subset\GL_2(\A_f)$ defined in \S\ref{sec:Casselman}. We are interested in the quotient
\[
X_1(q)=[\GL_2]/K_1(q)=\GL_2(\Q)\backslash\GL_2(\A)^1/K_1(q).
\]
Indeed, the cuspidal automorphic spectrum of conductor dividing $q$ (and with normalized central character) is naturally defined as functions on this space. 


We now write the above double quotient more classically. We use the more suggestive notation $\SL_2^\pm(\R)$, mentioned in \S\ref{sec:gp-decomp}, in place of $\GL_2(\R)^1$. From strong approximation, and the fact that $\GL_2(\Q)$ contains elements of negative determinant, we have
\[
\GL_2(\A)^1=\GL_2(\Q)\SL_2^\pm(\R)K_1(q)=\GL_2(\Q)\SL_2(\R)K_1(q).
\]
From this it follows that we may write $X_1(q)$ classically as $\Gamma_1(q)\backslash \SL_2(\R)$, where we have put
\[
\Gamma_1(q)=K_1(q)\cap\GL_2^+(\Q)=\left\{g=\begin{pmatrix} a & b\\ c & d\end{pmatrix}\in\SL_2(\Z): c\equiv d-1\equiv 0\mod q\right\}.
\]

We now descend further from $X_1(q)$ to the underlying modular curve (or surface) $Y_1(q)$, by quotienting out by the maximal connected compact subgroup of $\GL_2(\R)^1$. Recall from~\S\ref{sec:gp-decomp} the notation $\mathbf{K}_\infty^+=\SO_2(\R)$. We identify $\SL_2(\R)/\SO_2(\R)$ with the upper half-plane $\mathbb{H}$. Then the cuspidal automorphic spectrum can be naturally viewed as sections of line bundles on the hyperbolic surface
\[
Y_1(q)=X_1(q)/\mathbf{K}_\infty^+=\Gamma_1(q)\backslash\mathbb{H}=\overline{\Gamma_1(q)}\backslash\mathbb{H},
\]
where $\overline{\Gamma_1(q)}$ is the image of $\Gamma_1(q)$ in $\PSL_2(\Z)$. This is reviewed in the special case of weight one Maass forms in Appendix \ref{sec:STF}.


\section{Holomorphic forms contribution}
\label{sec:holomorphic-forms}

We prove in this section the following result:
\begin{prop}
\label{prop:counting-law-discrete}
We have, as $X$ grows to infinity,
\[
N_\mathcal{D}(X) = \frac{1}{3} \xi(2){\rm vol}_{\hat\mu_\mathfrak{F}}(\Omega_\mathcal{D}) X^3 + O(X^{2}).
\]
\end{prop}

\begin{proof}
Let $S_{\rm new}(k,N)$ be the set of holomorphic cuspidal newforms of weight $k$ and level~$\Gamma_1(N)$. The conductor of a newform $f \in S_{\rm new}(k,N)$ is $k^2N$. Then
\[
N_\mathcal{D}(X)=\sum_{k^2 N \leqslant X} \dim S_{\rm new} (k,N).
\]
To express the asymptotic dimension of $S_{\rm new} (k,N)$ we introduce the following arithmetical function. We let $s(N)$ be the multiplicative function satisfying
\[
s(p)=1-\frac{3}{p^2},\quad s(p^2)=1-\frac{3}{p^2}+\frac{3}{p^4},\quad s(p^\alpha)=\left(1-\frac{1}{p^2}\right)^3 \;\textrm{ for } \alpha\geq 3.
\]
Then one may deduce that
\[
\dim S_{\rm new} (k,N)= \frac{(k-1)N^2}{24}s(N)+O(N + N^\varepsilon k).
\]
An immediate calculation shows that $\mu \star s(N) \ll N^{-2}$. By \cite[Proposition 31]{martin_dimensions_2005}, we get
\[
\sum_{N\leqslant X} N^2 s(N) = c\frac{X^{3}}{3} + O(X^2),
\]
where, using the explicit description of $s(N)$,
\[
c = \prod_p \left( 1 - \frac{1}{p} \right) \left( 1 + \frac{s(p)}{p} + \frac{s(p^2)}{p^2} + \cdots  \right)=\prod_p \left( 1 - \frac{3}{p^3} + \frac{3}{p^6} - \frac{1}{p^9} \right) = \frac{1}{\zeta(3)^3}.
\]
Thus, 
\[
N_{\mathcal{D}}(X)= \sum_{k^2\leq X}\frac{k-1}{24}\sum_{N\leq X/k^2} N^2s(N) + O(X^2) = \frac13\frac{1}{\zeta(3)^3} X^3\sum_{k\geq 2}\frac{(k-1)/24}{k^6} + O(X^2).
\]

We have only to massage the main term into the desired form. Using the determination of $m^{\rm pl}_\infty(D_k)$ in \eqref{eq:Dk-Plancherel} as well as $\xi(2)=\zeta(2)\Gamma_\mathbb{R}(2)=\pi/6$, we have
\[
\frac{k-1}{24}=\frac{\pi}{6}\frac{k-1}{4\pi}=\xi(2)m^{\rm pl}_\infty(D_k).
\]
Inserting this into the above asymptotic gives
\[
\sum_{k^2N \leqslant X} \dim S_{\rm new} (k,N) + O(X^2) = \frac13 \xi(2) \frac{1}{\zeta(3)^3} \sum_{k\geq 2} \frac{m^{\rm pl}_\infty(D_k)}{c(D_k)^3} X^3+O(X^2).
\]
We recognize the leading term
\[
\frac{1}{\zeta(3)^3} \sum_{k\geq 2} \frac{m^{\rm pl}_\infty(D_k)}{c(D_k)^3}=\frac{1}{\zeta(3)^3}\int_\mathcal{D}\frac{\d m^{\rm pl}_\infty(\pi_\infty)}{c(\pi_\infty)^3}={\rm vol}_{\hat\mu_\mathfrak{F}}(\Omega_\mathcal{D}), 
\]
finishing the proof.
\end{proof}

\begin{remark}
The computations in the proof of Proposition \ref{prop:counting-law-discrete} do not critically depend on the exact value of the archimedean conductor $c(D_k)$ of the discrete series, but hold similarly with any value of $c(D_k)$, provided the final sum converges.
\end{remark}

\section{Maass forms contribution}
\label{sec:maass-forms}

We prove in this section the following result, which provides for the asymptotic count of the weight zero and weight one Maass forms. Throughout this section we shall let $k=0,1$.
\begin{prop}
\label{prop:counting-law-Maass}
We have, as $X$ grows to infinity and for all $\varepsilon >0$,
\[
N_{\mathcal{P}_k}(X) =  \frac{1}{3} \xi(2) {\rm vol}_{\hat\mu_{\mathfrak{F}}}(\Omega_{\mathcal{P}_k}) X^3 + O_\varepsilon(X^{3-\frac23 + \varepsilon} ).
\]
\end{prop}

Our approach is through a Tauberian argument applied to the conductor zeta function, as in Section \ref{sec:GL1-case}. The latter is defined, for $\Re(s)$ large enough, by
\[
Z^k(s)  = \sum_{\pi \in \mathfrak{F}_{\mathcal{P}_k}} c(\pi)^{-s}.
\]
Its analytic properties are described in the following proposition. We will thoroughly use the notation $\sigma = \Re(s)$.

\begin{prop}
\label{prop:analytic-properties-Maass}
The function $Z^k(s)$ satisfies the following properties:
\begin{enumerate}
\item it is holomorphic on $\sigma>3$;
\item it admits a meromorphic continuation to $\sigma>2$;
\item it has a unique pole in the right-half plane $\Re (s)>2$ located at $s=3$, which is simple and has residue $\xi(2) {\rm vol}_{\hat\mu_\mathfrak{F}}(\Omega_{\mathcal{P}_k})$;
\item it is of moderate growth in vertical strips. More precisely, for any $0<\delta<1$, we have the bound $Z^k(2+\delta+it) \ll (1+|t|)^{\frac{1}{2}}$.
\end{enumerate}
\end{prop}

An application of Theorem \ref{thm:Tauberian} to $Z^k$ then yields Proposition \ref{prop:counting-law-Maass}.

\subsection{Sieving spectral multiplicities}\label{sec:sieve}

Proposition \ref{prop:analytic-properties-Maass} relies on the Selberg trace formula, which we review in Appendix \ref{sec:STF}, and will require some preliminaries. We begin by a reformulation more suitable to the trace formula: 
\begin{lemma}
\label{lem:sieving}
We have the identity
\begin{equation}
\label{Z-function}
Z^k(s) = \sum_{q \geqslant 1} q^{-s} \sum_{d \mid q} \lambda\left( \frac{q}{d} \right) \sum_{\substack{\pi\in\mathfrak{F}_{\mathcal{P}_k} \\ c(\pi_f) \mid d }}  \dim \pi^{K_1(d)} c(\pi_\infty)^{-s}.
\end{equation}
\end{lemma}
\begin{proof}
We separate the sum into fixed values of $c(\pi_f)=q$, obtaining
\[
Z^k(s) =\sum_{q \geqslant 1} q^{-s}  Z^k(q,s), \quad\textrm{where}\quad Z^k(q,s) =\displaystyle \sum_{\substack{\pi \in\mathfrak{F}_{\mathcal{P}_k}\\ c(\pi_f) = q }}  c(\pi_\infty)^{-s}.
\]
We now introduce also a similar series, weighted with old form dimensions:
\begin{align}
\label{weighted-spectral-sum}
Z^{k, \rm old}(q,s) & = \sum_{\substack{\pi \in\mathfrak{F}_{\mathcal{P}_k} \\ c(\pi_f) \mid q }}  \dim \pi^{K_1(q)} c(\pi_\infty)^{-s}.
\end{align}
By Casselman's dimension formulas \cite{casselman_results_1973}, we have  $\dim \pi^{K_1(q)} = \tau(q/c(\pi_f))$, so that rewriting the series $Z^{k,\rm old}(q,s)$ as a sum over constant values of $c(\pi_f)$, we get
\[
Z^{k, \rm old}(q,s) = \sum_{d \mid q} \tau\left( \frac{q}{d} \right) \sum_{\substack{\pi \in\mathfrak{F}_{\mathcal{P}_k} \\ c(\pi_f) = d }}  c(\pi_\infty)^{-s} = \left(\tau \star Z^k (\cdot ,s)\right)(q). 
\]
Let $\lambda = \mu \star \mu$. By M\"obius inversion, we therefore conclude
\[
Z^k(q,s) = \left(\lambda \star Z^{k, \rm old}(\cdot, s)\right)(q) = \sum_{d \mid q} \lambda\left( \frac{q}{d} \right) Z^{k, \rm old}(d,s).
\]
Multiplying by $q^{-s}$ and summing over $q$, we conclude to the claimed equality \eqref{Z-function}. 
\end{proof}

By Remark \ref{rem:admissibility}, the function $h(\nu) = h_s(\nu) = c(\nu)^{-s}$ is admissible in the sense of Definition~\ref{defn-admissible}. We may therefore apply the Selberg trace formula Proposition \ref{prop:STF}, specialized to the congruence quotient $Y_1(q)$ with $h =h_s$. Noting that therefore $J^k_{Y_1(q),\mathrm{cusp}}(h) = Z^{k, \mathrm{old}}(q,s)$, we obtain
\[
Z^{k, \mathrm{old}}(q,s)=   J^k_{Y_1(q), \text{id}}(h) + J^k_{Y_1(q), \text{hyp}} (h) + J^k_{Y_1(q), \text{ell}}(h) + J^k_{Y_1(q), \text{para}}(h) -  J^k_{Y_1(q), \text{cont}}(h).
\]
Summing over the levels yield and sieving by Lemma \ref{lem:sieving}, we get (with obvious notations)
\begin{equation}
\label{eq:Z-splitting}
Z^k (s) = Z_{\text{id}}^k(s) + Z_{\text{hyp}}^k (s) + Z_{\text{ell}}^k(s) + Z_{\text{para}}^k(s)  -  Z_{\text{cont}}^k(s).
\end{equation}

We study separately each term and its analytic properties.

\subsection{Identity contribution}

We begin by addressing the identity contribution.
\begin{lemma}
\label{lem:Z-id}
The term $Z_{\rm id}^k(s)$ extends meromorphically up to $\sigma \geqslant 2+\delta$,  with a unique pole at $s=3$ which is simple and has residue $\xi(2){\rm vol}_{\hat{\mu}_\mathfrak{F}}(\Omega_{\mathcal{P}_k})$. Moreover, for all $\delta \in (0,1)$ and all~$t \in \R$, we have $Z_{\rm id}^k(2+\delta+it) \ll (1+|t|)^\frac{1-\delta}{2}$.
\end{lemma}

\begin{proof}
Recall that
\begin{align*}
J^0_{Y_1(q), \rm id}(h)&=\frac{\mathrm{vol}\, Y_1(d)}{4\pi} \int_\mathbb{R} c(\nu)^{-s} \nu \tanh(\pi \nu) \d\nu=\mathrm{vol}\,Y_1(d) \int_{\mathcal{P}_0}\frac{\d m^{\rm pl}_\infty(\pi_\infty)}{c(\pi_\infty)^s},\\
J^1_{Y_1(q), \rm id}(h)&=\frac{\mathrm{vol}\, Y_1(d)}{4\pi} \int_\mathbb{R} c(\nu)^{-s} \nu \coth(\pi \nu) \d\nu=\mathrm{vol}\, Y_1(d)\int_{\mathcal{P}_1}\frac{\d m^{\rm pl}_\infty(\pi_\infty)}{c(\pi_\infty)^s},
\end{align*}
where we have used \eqref{eq:P1-P2-Plancherel}. Now, we have
\[
\mathrm{vol}\, Y_1(d)=\mathrm{vol}\left(\PSL_2(\Z)\backslash\mathbb{H}\right)\cdot [\PSL_2(\Z):\overline{\Gamma_1(d)}].
\]
For $d\geqslant 3$ this can be expressed simply \cite[Theorem 4.2.5]{miyake_modular_1989} as
\begin{equation}
\label{eq:index}
\mathrm{vol}\, Y_1(d)=\mathrm{vol}\left(\PSL_2(\Z)\backslash\mathbb{H}\right)\cdot [\PSL_2(\Z):\overline{\Gamma_1(d)}]=\frac{\pi}{3} \cdot \frac{\varphi_2(d)}{2}=\xi(2)\varphi_2(d),
\end{equation}
where $\varphi_2$ is the arithmetical function defined in \eqref{eq:vol-K1(pf)}. 
Summing over the parameters in the expression \eqref{Z-function}, we get an arithmetic sum
\begin{equation*}
\label{arithmetic-sum}
\sum_{q \geqslant 1} q^{-s} \sum_{d \mid q} \lambda\left( \frac{q}{d} \right)  \mathrm{vol} \, Y_1(d)  = \sum_{d \geqslant 1} \frac{\mathrm{vol}\ Y_1(d)}{d^s} \sum_{e \geqslant 1}\frac{\lambda(e)}{e^s} =\xi(2)\frac{\zeta(s-2)}{\zeta(s)^3} + f(s),
\end{equation*}
\noindent which is convergent for $\sigma > 3$.  Here,  $f(s)$ accounts for the fact that \eqref{eq:index} does not hold for the exceptional cases $d \in \{1, 2\}$. The function $f(s)$ can be written as an entire and vertically bounded function multiplied by $\zeta(s)^{-2}$, hence is analytic for $\Re(s)>0$.   
 This yields
\begin{equation*}
\label{id-contribution-maass}
Z_{\text{id}}^k(s) = \xi(2) \frac{\zeta(s-2)}{\zeta(s)^3} \int_{\mathcal{P}_k}\frac{\d m^{\rm pl}_\infty(\pi_\infty)}{c(\pi_\infty)^s} + f(s) \int_{\mathcal{P}_k}\frac{\d m^{\rm pl}_\infty(\pi_\infty)}{c(\pi_\infty)^s} .
\end{equation*}
Moreover, by \cite[(10.10)]{montgomery_multiplicative_2006} we deduce that for all $\delta \in (0, 1)$, on $\sigma = 2 + \delta$ we have
\[
Z_{\rm id}^k (s) \ll (1+|s|)^\frac{1-\delta}{2}.\qedhere
\]
\end{proof}

\subsection{Eisenstein contribution}
\label{sec:Z-Eisenstein}

Next we address the contribution of the continuous spectrum to the trace formula, which is standard and which we recall for completeness.

\begin{lemma}
The term $Z_{\mathrm{cont}}^k(s)$ converges analytically for $\sigma \geqslant 2 + \delta$ for all $\delta >0$, and is uniformly bounded on vertical strips in this region.
\end{lemma}
\begin{proof}
The continuous contribution of the spectrum $Z_k^{\text{cont}}(s)$ is well-understood. Indeed, Huxley \cite{huxley_scattering_1984} provides an explicit formula for the determinant of the scattering matrix for~$Y_1(q)$. In fact, his setting is limited to $k=0$, but the $k=1$ case is identical, apart from the parity of the characters considered. He proves
\[
\varphi_k(s) = (-1)^\ell \left( \frac{A(q)}{\pi^{\kappa(q)}} \right)^{1-2s} \prod_{j=1}^{\kappa(q)} \frac{\Lambda(2-2s, \bar\chi_j)}{\Lambda(2s, \chi_j)}, 
\]
where $\ell$ is explicitly defined, $\kappa (q)$ is the number of inequivalent cusps in $Y_1(d)$, the $\chi_j$ are Dirichlet characters satisfying $\chi_j(-1)=(-1)^k$ with associated completed $L$-functions~$\Lambda(s, \chi_j)$, and
\[
A(q) = \prod_{\substack{q_1 q_2 \mid q}} \ \prod_{\substack{\chi_1 \mod q_1 \\ \chi_2 \mod q_2 }} q_1 q.
\]
Taking the log-derivative and examining the truncated integral arising in $Z^k_{\rm cont}$, we get 
\[
\frac{1}{4\pi} \int_{-T}^T \frac{\varphi'_k}{\varphi_k}(\tfrac{1}{2} + it) \d t = - \frac{T}{\pi} \log \left( \frac{A(q)}{\pi^{\kappa(q)}} \right) - \frac{1}{\pi}\sum_{j=1}^{\kappa(q)} \int_{-T}^T \frac{\Lambda'}{\Lambda}(1+2it, \chi_j) \d t.
\]

The second term splits into gamma factors and finite-part $L$-functions. By Stirling's formula, the main contribution is $T\log T$ for the gamma factors;  and by standard estimates it is also $T\log T$ on Dirichlet $L$-functions, see e.g. \cite[Chapter 5, Section 2]{risager_automorphic_2003}. Taking into account the size of the summation, we get a bound of $\kappa(q) T\log T$. We also have (see e.g. \textit{ibid.} Lemma 55)  the bound $\kappa(q) \ll q^{1+\varepsilon}$ and $\log A(q) \ll q^{1+\varepsilon}$.
We get that the truncated integral is bounded by 
\[
\frac{1}{4\pi} \int_{-T}^T \frac{\varphi'_k}{\varphi_k}(\tfrac{1}{2} + it) \d t  \ll q^{1+\varepsilon} T\log T.
\]

We can conclude by integration by parts that
\begin{align}
\label{eq:scattering-bound}
\frac{1}{4\pi} \int_{\mathbb{R}} \frac{\varphi'_k}{\varphi_k}(\tfrac{1}{2} + it) h(t) \d t & \ll q^{1+\varepsilon} \int_{\mathbb{R}} t^{2+\varepsilon}c(t)^{-\sigma - 1} \d t,
\end{align}
and this last integral converges for $\sigma > 0$. Moreover, the trace $\tr\,\Phi_k(\tfrac12)$ has been computed in~\cite[Lemma 2]{huxley_scattering_1984} and in particular satisfies $|\tr\,\Phi_k(\tfrac12)| \leqslant 2 q^{1/2}$. Since it is smaller than the bound obtained in \eqref{eq:scattering-bound}, we deduce
\begin{equation}
Z_{\mathrm{cont}}^k(s) \ll \sum_{q \geqslant 1} q^{-\sigma + 1 + \varepsilon} \int_{\mathbb{R}} c(\nu)^{-\sigma} \d\nu,
\end{equation}
so that $Z_{\rm cont}^k(s)$ is analytic for $\sigma \geqslant 2 + \delta$ and is uniformly bounded in vertical strips.
\end{proof}

\subsection{Hyperbolic contribution}
\label{sec:hyperbolic-contribution}

The hyperbolic contribution is the most delicate to bound. To prove the following result, we shall need to take advantage of the sum on levels.

\begin{prop}
The term $Z_{\rm hyp}^k (s)$ converges analytically for $\sigma \geqslant 2+\delta$ and is uniformly bounded in vertical strips in this region.
\end{prop}

\begin{proof}
The function $h(\nu) = c(\nu)^{-s}$ extends up to $|\Im(\nu)| \leqslant 1 - \varepsilon$ for all $\varepsilon>0$. We in particular deduce the estimate on the Fourier transform $g(\nu) \ll e^{-(1-\varepsilon) \nu}$ by the Paley-Wiener theorem.
The worst term in the hyperbolic contribution $J^k_{Y_1(q),\mathrm{hyp}}(h)$ comes from the value $\ell = 1$, so we can concentrate on this case. 

Using the above bound on $g$, and adding back the summation over $q, d$, we get 
\begin{align*}
Z_{\rm hyp}^k (s) & = \sum_{q \geqslant 1} q^{-s} \sum_{d \mid q} \lambda\left( \frac{q}{d} \right) \sum_{\substack{\gamma \in \mathcal{H}_d}} \sum_{\ell \geqslant 1} \frac{\log N_\gamma}{N_\gamma^{\ell/2} - N_\gamma^{-\ell/2}} g(\ell \log N_\gamma)\\
& \ll
\sum_{q \geqslant 1} q^{-\sigma} \sum_{d \mid q}\sum_{\substack{\gamma \in \mathcal{H}_d}}\frac{\log N_\gamma}{N_\gamma^{1/2}} e^{-(1-\varepsilon) \log N_\gamma},
\end{align*}
where $\mathcal{H}_d$ is the set of conjugacy classes of hyperbolic elements in $\Gamma_1(d)$ and $N_\gamma$ the norm of the associated geodesic, as recalled in Appendix \ref{sec:STF}. Note that if $\gamma$ is a conjugacy class in~$\Gamma_1(d)$, we necessarily have $d \mid \det(I-\gamma)$. Denote $c_\gamma(d)$ the multiplicity with which $\gamma$ appears in $\Gamma_1(d)$, namely
\[
c_\gamma(d) = |\{ x \in \Gamma_1(d) \backslash \SL_2(\Z) \ : \ x^{-1} \gamma x \in \Gamma_1(d)\}|.
\]
Switching summations, the above rewrites as
\[
Z_{\rm hyp}^k (s)   \ll \sum_{\gamma \in \mathcal{H}_1} N_\gamma^{-\frac{3}{2} + \varepsilon} \log N_\gamma \sum_{d \mid \det(\gamma-I)} c_\gamma(d) d^{-\sigma} \sum_{e \geqslant 1} e^{-\sigma}.
\]

For the innermost sums, we use the trivial bound $c_\gamma(d) \leqslant [\mathrm{SL}_2(\mathbb{Z}) : \Gamma_1(d)] \ll_\varepsilon d^{2+\varepsilon}$ so that 
\[
\sum_{d \mid \det(\gamma-I)} c_\gamma(d) d^{-\sigma} \ll N_\gamma^\varepsilon
\]

\noindent for all $\sigma \geqslant 2+\delta$ (we bound by the divisor function of $\det(I-\gamma) \ll N_\gamma$). The remaining~$e$-sum of course converges up to $\sigma \geqslant 1+\delta$.

It remains to justify the convergence of the sum over $\gamma$. We appeal to the Weyl law for closed geodesics \cite{bergeron_spectrum_2016}; it implies the estimate
\[
\sum_{\substack{\gamma\in \mathcal{H}_1 \\ N_\gamma \leqslant X}} \log N_\gamma \sim X, 
\]
\noindent and thus we deduce, by cutting in dyadic intervals $N_\gamma \sim Y \leqslant X$, the bound
\[
\sum_{\substack{\gamma \in \mathcal{H}_1 \\ N_\gamma \sim Y}} N_\gamma^{-\frac{3}{2} + \varepsilon} \log N_\gamma \ll_\varepsilon Y^{-\frac{3}{2}+\varepsilon } \sum_{\substack{\gamma \in \mathcal{H}_1 \\ N_\gamma \sim Y}}  \log N_\gamma \ll_\varepsilon Y^{-\frac{1}{2} + \varepsilon}.
\]
Altogether, $Z_{\rm hyp}^k (s)$ is analytic for $\sigma \geqslant 2+\delta$ and is bounded on vertical strips.
\end{proof}

\subsection{Proof of Proposition \ref{prop:analytic-properties-Maass}}

The number of elliptic elements is given by the formulas~\cite[Theorem 4.2.9]{miyake_modular_1989}, and we obtain that the contribution $Z_{\rm ell}^k(s)$ coming from the elliptic elements is uniformly bounded. The contribution $Z_{\rm para}^k(s)$ from the parabolic elements is smaller than the contribution of the continuous spectrum and is addressed using the bounds on $\kappa(q)$ already stated, see Section \ref{sec:Z-Eisenstein}. Finally, the functions~$Z_{\rm ell}^k(s)$ and $Z_{\rm para}^k(s)$ converge uniformly as long as $\sigma>1$. We deduce that they are  holomorphic for $\sigma > 1$ and bounded in vertical strips in this region.

Altogether, each term appearing in \eqref{eq:Z-splitting} continues analytically up to $\sigma>2$ and is uniformly bounded in vertical strips, except for $Z_{\rm id}^k$ that continues meromorphically in this region with a simple pole at $s=3$ and explicit vertical growth. Hence, $Z^k(s)$ satisfies the claimed properties, and that ends the proof of Proposition \ref{prop:analytic-properties-Maass}, and therefore of \ref{prop:counting-law-Maass}.\qed

\begin{appendix}
\section{Selberg trace formula}\label{sec:STF}
We recall the Selberg trace formula for weight zero and weight one Maass forms. 

\subsection{Notation}
Let $\Gamma$ be an arbitrary lattice in $\SL_2(\R)$ and put $Y=\Gamma\backslash\mathbb{H}$. We now introduce some notation related to the spectral and geometric invariants of $Y$ involved in the trace formulae we study.

We endow $Y$ with the hyperbolic metric induced from the Poincar\'e metric $\d x\d y/y^2$ on $\mathbb{H}$. We denote by ${\rm vol}\, Y$ the volume of $Y$ with respect to the corresponding measure.

For $k=0,1$, let $\Phi_k$ be the scattering matrix of the weight $k$ Eisenstein series, as defined in~\cite[p.~87]{iwaniec_spectral_2002} in weight 0 and \cite[pp. 368-369]{hejhal_selberg_1983} for more general weights. Let $\varphi_k$ its determinant.

Denote by $\mathcal{H}_{\rm prim}$ the set of conjugacy classes of primitive hyperbolic elements in $\Gamma$. Any~$\gamma\in~\mathcal{H}$ can be conjugated in $\SL_2(\R)$ to a matrix of the form ${\rm diag}(\lambda,\lambda^{-1})$, where~$|\lambda|>1$. We denote by $N_\gamma=\lambda^2$ and ${\rm sgn}(\gamma)={\rm sgn}(\lambda)$. Denote by $\mathcal{E}_{\rm prim}$ the set of conjugacy classes of primitive elliptic elements in $\Gamma$, and, for every class $\gamma\in\mathcal{E}$, we denote by $m=m_{\gamma}$ its associated order and usually drop the subscript from notation for easier readability. Let $\kappa$ be the number of cusps of $Y$. We also denote by $\psi(s)=\Gamma'(s)/\Gamma(s)$ the classical digamma function. 

We shall use the following class of test functions, taken from \cite[(1.63)]{iwaniec_spectral_2002}. 

\begin{definition}\label{defn-admissible} A function $h(\nu)$ will be called \emph{admissible} if it satisfies the following conditions: there is $\delta>0$ such that
\begin{itemize}
\item[--] it is even on $\mathbb{R}$; 
\item[--] it extends analytically to the strip $|\Im(\nu)| \leqslant \frac{1}{2}+\delta$; 
\item[--] it verifies $h(\nu)\ll (1+|\nu|)^{-2-\delta}$ in the strip $|\Im(\nu)| \leqslant \frac{1}{2}+\delta$.
\end{itemize}
\end{definition}

For an admissible function $h$ we let
\[
g(x)=\frac{1}{2\pi}\int_\R h(\nu)e^{i\nu x}\d x
\]
denote its Fourier transform.

\subsection{Weight zero trace formula}\label{sec:weight-zero-STF}
We now state the Selberg trace formula in the weight zero case. Let $\mathcal{B}_0(Y)=\{\phi_\nu\}_\nu$ be a basis of $L^2_{\rm cusp}(Y)$ consisting of Maass cusp forms with
\[
(\Delta+(1/4+\nu^2))\phi_\nu=0,\quad\textrm{where }\; \nu\in\R\cup i(-1/2,1/2).
\]
The following theorem is due to Selberg; see also \cite[Theorem 10.2]{iwaniec_spectral_2002}.

\begin{prop}[Selberg Trace Formula -- weight zero]\label{prop:STF}
For all admissible function $h$,
\[
J^0_{Y,\mathrm{cusp}}(h) + J^0_{Y,\mathrm{cont}}(h) = J^0_{Y,\mathrm{id}}(h) + J^0_{Y,\mathrm{hyp}} (h) + J^0_{Y,\mathrm{ell}}(h) + J^0_{Y,\mathrm{para}}(h), 
\]
where
\[
J^0_{Y,\mathrm{cusp}}(h)  = \sum_{\phi_\nu\in \mathcal{B}_0(Y)} h(\nu) \quad\textrm{and}\quad
J^0_{Y,\mathrm{cont}}(h) =  \frac{1}{4\pi} \int_{\mathbb{R}} h(\nu) \frac{-\varphi_0'}{\varphi_0}(\tfrac{1}{2}+i\nu) \d \nu + \frac{h(0)}{4} \mathrm{tr}(\Phi_0(\tfrac{1}{2})) 
\]
are the cuspidal and Eisenstein series contributions,
\[
J^0_{Y,\mathrm{id}}(h)  = \frac{\mathrm{vol} \ Y}{4\pi} \int_\mathbb{R} h(\nu) \nu \mathrm{tanh} (\pi \nu) \d \nu
\]
is the identity contribution, and
\begin{align*}
J^0_{Y,\mathrm{hyp}}(h) & = \sum_{\gamma \in \mathcal{H}_{\rm prim}} \sum_{\ell \geqslant 1} \frac{g(\ell \log N_\gamma)}{N_\gamma^{\ell/2} - N_\gamma^{-\ell/2}} \log N_\gamma \\
J^0_{Y,\mathrm{ell}}(h) & = \sum_{\mathcal{E}_{\rm prim}} \sum_{\ell = 1}^{m-1} \left( 2m \sin \tfrac{\pi \ell}{m} \right)^{-1}\int_{\mathbb{R}} h(r) \frac{\mathrm{cosh}\left( \pi r( \left( 1 - \tfrac{2\ell}{m} \right)\right)}{\mathrm{cosh}\pi r} \d r \\
J^0_{Y,\mathrm{para}}(h) & = - \kappa\left(\frac{1}{2\pi} \int_{\mathbb{R}} h(r) \psi(1+ ir) \d r + \frac{h(0)}{4} - g(0) \log 2\right)
\end{align*}
are the contributions from the hyperbolic, elliptic, and parabolic classes.
\end{prop}

\subsection{Weight one trace formula}\label{sec:weight-one-STF}

We now turn to the Selberg trace formula for weight one Maass forms on $Y=\Gamma\backslash\mathbb{H}$. Before doing so, we provide some recollections of the weight one setting.

Let $M_1(Y)$ denote the space of complex-valued functions $\phi$ on $\mathbb{H}$ such that
\[
\phi(\gamma\cdot z)=j_\gamma(z) \phi(z)\quad\textrm{for all } \gamma\in \Gamma,
\]
where $j_\gamma(z)=(cz+d)/|cz+d|$ for $\gamma=\left(\begin{smallmatrix} a & b \\ c & d\end{smallmatrix}\right)$, and which are of moderate growth as $z$ approaches the cusps of $\Gamma$. We may view a function $\phi\in M_1(Y)$ as a section of a line bundle~$\mathcal{L}$ over $Y$ given by $\mathcal{L}=\Gamma\backslash(\mathbb{H}\times \C)$, with $\gamma\cdot (z,s)=(\gamma\cdot z, j_\gamma(z)s)$. Let $L^2(Y, \mathcal{L})$ denote the Hilbert space completion of $M_1(Y)$ with respect to the inner product $\int_Y \phi_1(z)\overline{\phi_2(z)} y^{-2}dxdy$. Let $L^2_{\rm cusp}(Y, \mathcal{L})$ denote the closed subspace of $L^2(Y, \mathcal{L})$ consisting of $\phi$ which vanish at the cusps.

Recall the definition of the weight one Laplacian
\[
\Delta_1=y^2\left(\frac{\partial^2}{\partial x^2}+\frac{\partial^2}{\partial y^2}\right)-iy\frac{\partial}{\partial x}.
\]
Then $-\Delta_1$ is a self-adjoint operator on $L^2(Y,\mathcal{L})$, preserving the cuspidal subspace $L^2_{\rm cusp}(Y, \mathcal{L})$, with spectrum in the interval $[1/4,\infty)$. A \textit{weight one Maass cusp form} is then any eigenfunction of $\Delta_1$ in $L^2_{\rm cusp}(Y, \mathcal{L})$. 

Let $\mathcal{B}_1(Y) = \{\phi_\nu\}_\nu$ be a basis of $L^2_{\rm cusp}(Y,\mathcal{L})$ consisting of weight one cusp forms with
\[
(\Delta_1+(1/4+\nu^2))\phi_\nu=0, \quad\textrm{where }\; \nu\in\R.
\]
The following formula can be found in \cite[Theorem 6.3]{hejhal_selberg_1983}; see also \cite{akiyama_selberg_1988}.

\begin{prop}[Selberg Trace Formula -- weight 1]\label{prop:STF2}
For all admissible function $h$,
\[
J^1_{Y,\mathrm{cusp}}(h) + J^1_{Y,\mathrm{cont}}(h) = J^1_{Y,\mathrm{id}}(h) + J^1_{Y,\mathrm{hyp}} (h) + J^1_{Y,\mathrm{ell}}(h) + J^1_{Y,\mathrm{para}}(h), 
\]
where
\[
J^1_{Y,\mathrm{cusp}}(h)  = \sum_{\phi_\nu\in \mathcal{B}_1(Y)} h(\nu) \quad\textrm{and}\quad
J^1_{Y,\mathrm{cont}}(h) =  \frac{1}{4\pi} \int_{\mathbb{R}} h(\nu) \frac{-\varphi_1'}{\varphi_1}(\tfrac{1}{2}+i\nu) \d \nu +\frac{h(0)}{4} \mathrm{tr}(\Phi_1(\tfrac{1}{2}))
\]
are the cuspidal and Eisenstein series contributions,
\[
J^1_{Y,\mathrm{id}}(h)  = \frac{\mathrm{vol} \ Y}{4\pi} \int_\mathbb{R} h(\nu) \nu \coth (\pi \nu) \d \nu
\]
is the identity contribution, and
\begin{align*}
J^1_{Y,\mathrm{hyp}}(h) & = \sum_{\gamma \in \mathcal{H}_{\rm prim}} \sum_{\ell \geqslant 1} \frac{{\rm sgn}(\gamma)\log N_\gamma}{N_\gamma^{\ell/2} - N_\gamma^{-\ell/2}}  g(\ell \log N_\gamma)\\
J^1_{Y,\mathrm{ell}}(h) & = \sum_{\mathcal{E}_{\rm prim}} \sum_{\ell = 1}^{m-1} \left( 2m \sin \tfrac{\pi \ell}{m} \right)^{-1} \left( \int_{\mathbb{R}} h(r) \frac{\mathrm{sinh}\left( \pi r( \left( 1 - \tfrac{2\ell}{m} \right)\right)}{\mathrm{sinh}\pi r} \d r - i h(0) \right) \\
J^1_{Y,\mathrm{para}}(h) & = -\kappa\left(\frac{1}{2\pi} \int_{\mathbb{R}} h(r) \psi(1+ ir) \d r +\frac{h(0)}{4} - g(0) \log 2 + \int_0^\infty g(u) \frac{1- \mathrm{cosh}(u/2)}{2\mathrm{sinh}(u/2)} \d u\right)
\end{align*}
are the contributions from the hyperbolic, elliptic, and parabolic classes.
\end{prop}

\section{Tauberian theorem}
\label{sec:tauberian}

Let $\mathfrak{X}$ be a discrete set and $c$ a positive function on $\mathfrak{X}$. The following Tauberian theorem provides asymptotics for $N(X) = \{x \in \mathfrak{X} \ : \ c(x) \leqslant X \}$ from the analytic properties of the associated zeta function, defined by
\[
Z_{\mathfrak{X}}(s) = \sum_{x \in \mathfrak{X}} c(x)^{-s}.
\]

\begin{theorem}\label{thm:Tauberian}
Suppose there exist constants $\alpha, \beta, \kappa \in \mathbb{R}_+$, with 
$\beta>\alpha$, such that $Z_{\mathfrak{X}}(s)$ satisfies the following properties:
\begin{itemize}
\item it extends meromorphically to $\Re(s) > \alpha$,
\item it has a unique pole at $s=\beta$, which is simple,
\item it has at most polynomial growth in every fixed closed vertical strip within $\Re(s)>\alpha$ and satisfies $Z_{\mathfrak{X}}(s) \ll (1+|s|)^{\kappa}$ on $\Re(s) = \alpha+\delta$, for all small enough $\delta >0$.
\end{itemize}
Then, for all $\varepsilon>0$,
\[
N(X) = \underset{s=\beta}{\mathrm{Res}} \, Z_{\mathfrak{X}}(s)  \cdot \frac{X^\beta}{\beta}  + O_{\varepsilon}\left(X^{\beta-\frac{\beta-\alpha}{\kappa  + 1} + \varepsilon}\right). 
\]
\end{theorem}


\begin{remark}
 Under its assumptions, this result improves upon \cite[Theorem A.1]{chambert-loir_igusa_2010} by making error terms explicit, and improving them upon those deduced from \textit{loc. cit}. We could allow for higher order poles, which would amount to picking up extra powers of $\log X$.
 \end{remark}

\begin{proof}

Introduce a smooth function $\phi \in C_c^\infty(\tfrac12, 2)$ normalized so that $\int_0^{\infty}\phi(x) \frac{\mathrm{d}x}{x} = 1$. For~$T \geqslant 2$, define~$\phi_T(x) = T\phi(x^T) \in  C_c^\infty(2^{-1/T}, 2^{1/T})$. Its Mellin transform $\widehat{\phi}_T(s) = \widehat{\phi}(s/T)$ is entire and satisfies, for $s$ in every fixed vertical strip,
\begin{equation}
\label{decay}
\widehat{\phi}_T(s) \ll_A \left( 1 + \frac{|s|}{T} \right)^{-A}.
\end{equation}

Let $\mathbf{1}$ the the characteristic function of $(0,1)$, so that $\widehat{\mathbf{1}}(s) = s^{-1}$ is its Mellin transform. We regularize $\mathbf{1}$ by taking the convolution $\mathbf{1}_T =  \mathbf{1} \star \phi_T$, in particular $\widehat{\mathbf{1}_T} = \widehat{\mathbf{1}}\cdot \widehat{\phi}_T$. Note that we have $0 \leqslant \mathbf{1}_T \leqslant 1$ on $\mathbb{R}_+$ and that $\mathbf{1}_T = \mathbf{1}$ outside $(2^{-1/T}, 2^{1/T})$.
Executing the Mellin inversion $\mathcal{M}^{-1} \mathcal{M} (\mathbf{1}_T) = \mathbf{1}_T$ yields
\begin{align}
\label{mellin-inverse}
\mathbf{1}_T(x) &  = \frac{1}{2i\pi} \int_{(\sigma)} x^{-s} \widehat{\phi}\left( \frac{s}{T} \right) \frac{\mathrm{d}s}{s}.
\end{align}


Instead of the sharp count $N(X)$, we study
\[
N_T(X) = \sum_{x \in \mathfrak{X}} \mathbf{1}_T \left( \frac{c(x)}{X}\right).
\]
By the above-mentioned properties of $\phi_T$, the smooth count $N_T(X)$ is related to $N(X)$ by 
\begin{equation}
\label{eq:over-under-count}
N_{T}(X2^{-1/T}) \leqslant N(X) \leqslant N_{T}(X2^{1/T}).
\end{equation}


Summing \eqref{mellin-inverse} over $x \in \mathfrak{X}$ we deduce that, for $\sigma >\beta$,
\[
N_T(X)= \frac{1}{2i\pi} \int_{(\sigma)} Z_{\mathfrak{X}}(s) X^s \widehat{\phi}\left( \frac{s}{T} \right) \frac{\mathrm{d}s}{s}.
\]


The function $Z_{\mathfrak{X}}(s)$ is assumed to have meromorphic continuation, with at most polynomial growth in every fixed closed vertical strip, up to $\Re(s) > \alpha$. The vertical rapid decay~\eqref{decay} of $\widehat{\phi}_T$ allows to move the contour up to $\alpha+\delta$, passing the pole located at $s=\beta$. We get
\[
N_T(X) = \mathrm{M}_T(X) + \mathrm{E}_T(X),
\]
where
\[
\mathrm{M}_T(X)  = \underset{s=\beta}{\mathrm{Res}} \, Z_{\mathfrak{X}}(s) \cdot  \widehat{\phi}\left(\frac{\beta}{T}\right) \frac{X^\beta}{\beta}  \qquad \mathrm{and} \qquad
\mathrm{E}_T(X)  = \int_{(\alpha+\delta)} Z_{\mathfrak{X}}(s) X^s \widehat{\phi}\left( \frac{s}{T} \right) \frac{\mathrm{d}s}{s}.
\]
Since $\widehat{\phi}_T = 1 + O(T^{-1})$, we conclude
\begin{equation*}
\label{main-term}
\mathrm{M}_T(X) = \underset{s=\beta}{\mathrm{Res}} \, Z_{\mathfrak{X}}(s) \cdot \frac{X^\beta}{\beta}  + O\left(\frac{X^\beta}{T} \right).
\end{equation*} 
The remaining integral $\mathrm{E}_T(X)$ can be easily dominated by using the assumption on the vertical growth of $Z_{\mathfrak{X}}(s)$:
\begin{equation*}
\label{E}
\mathrm{E}_T(X) \ll \int_{(\alpha+\delta)} (1+|s|)^{\kappa} \left( 1+\frac{|s|}{T} \right)^{-A} \frac{\mathrm{d}s}{|s|}X^{\alpha+\delta} \ll X^{\alpha+\delta} T^{\kappa}.
\end{equation*}

Inserting these asymptotics in \eqref{eq:over-under-count}, we deduce $N_T(X2^{1/T}) -N(X) \ll X^\beta/T + X^{\alpha + \delta} T^\kappa$. Equating these two error terms, we obtain an optimal choice of $T = X^{(\beta-\alpha)/(\kappa + 1)}$, and we are left with the counting law
\[
N(X) = \underset{s=\beta}{\mathrm{Res}} \, Z_{\mathfrak{X}}(s) \cdot\frac{X^\beta}{\beta} + O\left(X^{\beta-\frac{\beta-\alpha - \delta}{\kappa  + 1}}\right).
\]
This finishes the proof.\end{proof}

\section{Reflections on alternative approaches}
\label{sec:alternative-approaches}
In this section we discuss some issues regarding smoothing and spectral inversion of test functions, and comment extensively on the recent work of Jana--Nelson \cite{jana_analytic_2019}. Since this appendix is purely expository and comparative, we shall work in the more general setting of the group $G=\GL_n$ over $\Q$. 
\subsection{Archimedean conductor via approximate invariance}
In \cite{jana_analytic_2019}, Jana and Nelson observe that the archimedean conductor of a generic irreducible $\pi\in G(\R)^{\wedge,{\rm gen}}$, as defined by Iwaniec--Sarnak, can be captured by its approximate invariance properties. More precisely, for $X\geq 1$ and $\tau\in (0,1)$, we denote
\[
K_1(X,\tau)=\left\{\begin{pmatrix} a & b\\ c & d\end{pmatrix}\in\GL_n(\R)\ \Bigg\vert \; a\in\GL_{n-1}(\R), d\in\GL_1(\R), \begin{aligned} |a-I_{n-1}|<\tau,\;\;&\qquad |b|<\tau\\ X|c|<\tau,\;\; &|d-1|<\tau \end{aligned} \right\},
\]
a compact subset of $\GL_n(\R)$. Here, $|\cdot |$ denotes any fixed norm on the appropriate set of matrices. They show that 
\begin{enumerate}
\item\label{JN1} for every $\delta>0$ there is $\tau>0$ such that every $\pi\in G(\R)^{\wedge,{\rm gen}}$ has $\delta$-approximate invariants by $K_1(c_{\rm IS}(\pi),\tau)$; 
\item\label{JN2} there is $\delta>0$ such that for every $\tau>0$ there is $\mu>0$ such that \textit{no} $\pi\in G(\R)^{\wedge,{\rm gen}}$ has $\delta$-approximate invariants by $K_1(\mu c_{\rm IS}(\pi),\tau)$.
\end{enumerate}
See Theorems 1 and 2 of \textit{loc. cit.} for the precise statements.\footnote{There is a dependence on these quantifiers on another variable $\theta\in [0,1/2)$ (as well as the $n$ in $\GL_n$) which measures how non-tempered $\pi$ can be. For our applications, we shall apply these properties to $\pi$ which are local components at infinity of cuspidal automorphic representations. The latter are known not only to be generic but also to have parameters whose real part are bounded away from $1/2$. Indeed the bounds of Luo-Rudnick-Sarnak show that $\theta=1/2-(n^2+1)^{-1}$ is admissible. Thus, finally, the values of $\delta,\tau,\mu$ depend only on $n$.} For $\GL_2$, this is an archimedean version of the result of Casselman recalled in \S\ref{sec:Casselman}, later extended to $\GL_n$ by Jacquet--Piatetski-Shapiro--Shalika \cite{jacquet_conducteur_1981}, which characterizes the local $p$-adic conductor via the invariants by the subgroup $K_{1,p}(f)$ of \eqref{eq:K1p-def}.

One could in turn seek to \textit{define} the archimedean conductor by the almost invariance properties of $\pi$ under $K_1(X,\tau)$ listed above. For this we would take a $\delta_0>0$ guaranteed by property \eqref{JN2}, which gives rises to $\tau_0>0$ by property \eqref{JN1}, and then to a $\mu_0>0$ by property \eqref{JN2} again. In this case, any triple $(\delta,\tau,\mu)$ with $\delta\geq\delta_0$, $\tau\leq\tau_0$, $\mu\leq\mu_0$ would continue to satisfy properties \eqref{JN1} and \eqref{JN2}, since they lead to weaker conditions. One would then define $c_{\rm JN}(\pi)$ to be the infimum of the $X\geq 1$ such that $\pi$ admits $\delta_0$-approximate invariants by $K_1(X,\tau_0)$. Property \eqref{JN1} would then amount to the inequality $c_{\rm JN}(\pi)\le c_{\rm IS}(\pi)$, while property~\eqref{JN2} would amount to the inequality $\mu_0 c_{\rm IS}(\pi)\leq c_{\rm JN}(\pi)$. In other words,
\[
c_{\rm JN}(\pi)\asymp c_{\rm IS}(\pi).
\]
This is similar to our condition \eqref{cond3} in Definition \eqref{defn-c}.

\subsection{Use in the trace formula}

A definition of archimedean conductor as suggested above lends itself quite naturally to use in the Arthur--Selberg trace formula (or other types of trace formulae, such as that of Kuznetsov), as has been powerfully demonstrated in the recent thesis of Jana \cite{jana_applications_2020} and subsequent outgrowths. In this usage, one takes as an archimedean test function a smoothened characteristic function of $K_1(X,\tau_0)$, see \cite[Section 8]{jana_analytic_2019}. The latter can be viewed as a bump function about the origin in the Kirillov model of $\pi$; it is the \textit{analytic new vector} in the title of \cite{jana_analytic_2019}. 

In this way, one captures in a much more direct way the condition on $\pi$ of being of archimedean conductor $c_{\rm JN}$ up to $X$ --- more direct than, say, by integrating a Paley-Wiener localizing function $h_\mu$ around $\mu\in\mathfrak{a}^*$ over $c_{\rm IS}\leq X$, as was done for every type of continuous series of representations in \cite{brumley_counting_2018}, following the template of \cite{duistermaat_spectra_1979} for the spherical Weyl law. 

To appreciate the power of the Jana--Nelson approach we highlight two striking advantages:
\begin{enumerate}
\item the first is to observe that $\delta_0$-approximate invariance by $K_1(X,\tau_0)$ captures the condition $c_{\rm JN}(\pi)\leq X$ \textit{regardless of how $\pi$ sits in the Langlands classification}. By contrast, in \cite{brumley_counting_2018}, the argument proceeded only after first fixing discrete induction data;
\item the second, related, observation is that the question of constructing \textit{explicit} test functions for use in the Arthur--Selberg trace function which capture the given archimedean condition is naturally built into this approach. By contrast, in \cite{brumley_counting_2018}, for $n\geq 3$, the authors limited themselves to counting the universal family for $\GL_n$ Maass forms (so spherical at infinity) precisely because of the lack of explicit control on the test functions furnished by the Paley--Wiener theorem of Clozel--Delorme. While it seems possible to use a more direct inversion formula in terms of weighted orbital integrals, which would eliminate the spherical hypothesis, the explicit test function on the group of Jana--Nelson, analogous to those at the finite places, remains quite appealing.
\end{enumerate}

By contrast, it is not clear (to the authors) how one would remove the oldform dimensions at the archimedean place that one would incur through a use of analytic new vectors in the trace formula, as one does at the finite places (see \S\ref{sec:sieve}) when one ``sieves for newforms''. This is closely related to Remark \ref{rem:arch-int}, since the proof of \eqref{eq:local-meas-volume} in \cite[Section 6.1]{brumley_counting_2018} could be successfully mimicked given such an archimedean sieving procedure.

\subsection{Smoothing}
Regardless of which approach one follows to smoothly capture the bounded archimedean conductor condition, when it comes to the desired result for understanding the asymptotic size of $|\mathfrak{F}(Q)|$ one can either content oneself with a smooth count or attempt to convert the smooth count to a sharp count on $|\mathfrak{F}(Q)|$. We compare the two below.

\subsubsection{Smoothly counting}
Leaving the count as a smooth one would alleviate many of the technical problems encountered in \cite{brumley_counting_2018}, as it would in fact do for any (higher rank) Weyl law, going back to \cite{duistermaat_spectra_1979}. Moreover, in practice, it is often just as useful to have a smoothly weighted count as it is to have a sharp count.

The archimedean test functions of Jana--Nelson would appear to be well-suited for obtaining a smoothened version of Theorem \ref{thm}, and would even allow for a treatment using only compactly supported test functions. We have not pursued this.

\subsubsection{Sharply counting}
Asking for a sharp count on $|\mathfrak{F}(Q)|$ requires a desmoothing procedure. Note that even if one does not require splitting $\pi$ into representation type, one nevertheless must run the trace formula one level at a time.

When this desmoothing is done for every given level, the error terms one incurs in certain ranges can only yield logarithmic gains. For example, when the level $q$ is large and the archimedean conductor is small, one can only obtain a savings of size $(\log q)^{-1}$; see \cite[\S 3.1]{brumley_counting_2018} for a more in depth discussion. It is precisely at this scale (bounded archimedean conductor) where the precise shape of the archimedean test function becomes immaterial. The work of Jana--Nelson therefore does not seem to alleviate this difficulty.

The other possibility is to desmooth the count only once, after having executed the sum on levels. The approach that we follow in this paper can be interpreted in this way. Indeed, one can smoothly approximate the archimedean test function $c(\nu)^{-s}$ by a Paley--Wiener function of exponential type $R$, and allow $R$ tend to infinity as the final step. It is not clear whether the framework of Jana--Nelson can be similarly adapted to give a sharp count.
\end{appendix}

\subsection*{Acknowledgments} 

Didier Lesesvre was a visiting scholar at the Chinese University of Hong Kong during part of this work, and is particularly grateful to Charles Li for his support and interesting discussions.

\bibliographystyle{acm}
\bibliography{biblio.bib}

\begin{thebibliography}{10}

\bibitem{akiyama_selberg_1988}
{\sc Akiyama, S.}
\newblock Selberg trace formula for odd weight, {I}.
\newblock {\em Proc. Jap. Acad. A 64}, 9 (1988), 341--344.

\bibitem{bergeron_spectrum_2016}
{\sc Bergeron, N.}
\newblock {\em The spectrum of hyperbolic surfaces}.
\newblock Universitext. Springer, EDP Sciences, 2016.

\bibitem{booker_l-functions_2015}
{\sc Booker, A.}
\newblock L-functions as distributions.
\newblock {\em Math Ann. 363}, 1-2 (2015), 423--454.

\bibitem{brumley_counting_2018}
{\sc Brumley, F., and Mili\'{c}evi\'{c}, D.}
\newblock Counting cusp forms by analytic conductor.
\newblock arXiv:1805.00633.

\bibitem{casselman_results_1973}
{\sc Casselman, W.}
\newblock On {Some} {Results} by {Atkin} and {Lehner}.
\newblock {\em Math Ann. 201\/} (1973), 301--314.

\bibitem{chambert-loir_igusa_2010}
{\sc Chambert-Loir, A., and Tschinkel, Y.}
\newblock Igusa integrals and volume asymptotics in analytic and adelic
  geometry.
\newblock {\em Confluentes Math 2}, 3 (2010), 351--429.

\bibitem{C-PS}
{\sc Cogdell, J., and Piatetski-Shapiro, I.}
\newblock {\em The arithmetic and spectral analysis of {P}oincaré series}.
\newblock Perspectives in Mathematics. Academic Press, Boston, MA, 1990.

\bibitem{conrey_integral_2005}
{\sc Conrey, J.~B., Farmer, D.~W., Keating, J.~P., Rubinstein, M.~O., M., and
  Snaith, N.~C.}
\newblock Integral moments of {L}-functions.
\newblock {\em Proc. London Math Soc. 91}, 3 (2005), 33--104.

\bibitem{Drinfeld}
{\sc Drinfeld, V.~G.}
\newblock Number of two-dimensional irreducible representations of the
  fundamental group of a curve over a finite field.
\newblock {\em Functional Analysis and its Applications 15}, 4 (1982),
  294--295.

\bibitem{duistermaat_spectra_1979}
{\sc Duistermaat, J.~J., Kolk, J. A.~C., and Varadarajan, V.~S.}
\newblock Spectra of compact locally symmetric manifolds of negative curvature.
\newblock {\em Invent. Math 52}, 1 (1979), 27--93.

\bibitem{finis_arthur-selberg_2011}
{\sc Finis, T., and Lapid, E.}
\newblock On the {Arthur}-{Selberg} trace formula for {GL}(2).
\newblock {\em Groups, Geom. Dyn. 5\/} (2011), 367--391.

\bibitem{finis_remainder_2019}
{\sc Finis, T., and Lapid, E.}
\newblock On the remainder term of the {Weyl} law for congruence subgroups of
  {Chevalley} groups.
\newblock arXiv:1908.06626.

\bibitem{hejhal_selberg_1983}
{\sc Hejhal, D.~A.}
\newblock {\em The {Selberg} {Trace} {Formula} for {PSL}(2,{$\mathbb{R}$})}.
\newblock Lecture {Notes} in {Math}. Springer, 1983.

\bibitem{hoffstein_siegel_1995}
{\sc Hoffstein, J., and Ramakrishnan, D.}
\newblock Siegel {Zeros} and {Cusp} {Forms}.
\newblock {\em Int. Math. Res. Not 6\/} (1995).

\bibitem{huxley_scattering_1984}
{\sc Huxley, M.~N.}
\newblock Scattering matrix for congruence subgroups.
\newblock In {\em Modular forms}. 1984, pp.~141--156.

\bibitem{iwaniec_spectral_2002}
{\sc Iwaniec, H.}
\newblock {\em Spectral {Methods} of {Automorphic} {Forms}}.
\newblock AMS, Providence, Rhode Island, 2002.

\bibitem{iwaniec_perspectives_2000}
{\sc Iwaniec, H., and Sarnak, P.}
\newblock Perspectives on the {Analytic} {Theory} of {L}-functions.
\newblock {\em GAFA Geom. Funct. Anal\/} (2000), 705--741.

\bibitem{jacquet_conducteur_1981}
{\sc Jacquet, H., Piatetski-Shapiro, G., and Shalika, J.}
\newblock Conducteur des représentations du groupe linéaire.
\newblock {\em Math Ann. 256}, 2 (1981), 199--214.

\bibitem{jana_applications_2020}
{\sc Jana, S.}
\newblock Applications of analytic newvectors for
  {$\mathrm{GL}_n(\mathbb{R})$}.
\newblock arXiv: 2001.09640.

\bibitem{jana_analytic_2019}
{\sc Jana, S., and Nelson, P.~D.}
\newblock Analytic newvectors for {$\mathrm{GL}_n(\mathbb{R})$}.
\newblock arXiv: 1911.01880.

\bibitem{knightly_traces_2006}
{\sc Knightly, A., and Li, C.}
\newblock {\em Traces of {Hecke} operators}.
\newblock No.~133 in Math. {Surv.} and {Mono.} AMS, 2006.

\bibitem{lang_sl2R_1985}
{\sc Lang, S.}
\newblock {\em {${\mathrm{SL}_2(\mathbb{R})}$}}.
\newblock No.~105 in Graduate {Texts} in {Mathematics}. Springer-Verlag, New
  York, 1985.

\bibitem{lesesvre_counting_2020}
{\sc Lesesvre, D.}
\newblock Counting and equidistribution for quaternion algebras.
\newblock {\em Math Z. 295\/} (2020), 129--159.

\bibitem{lindenstrauss_existence_2007}
{\sc Lindenstrauss, E., and Venkatesh, A.}
\newblock Existence and {Weyl}’s law for spherical cusp forms.
\newblock {\em GAFA Geom. Funct. Anal. 17}, 1 (2007), 220--251.

\bibitem{martin_dimensions_2005}
{\sc Martin, G.}
\newblock Dimensions of the spaces of cusp forms and newforms on
  {$\Gamma_0({N})$} and {$\Gamma_1({N})$}.
\newblock {\em Journal of Number Theory 112}, 2 (2005), 298--331.

\bibitem{miyake_modular_1989}
{\sc Miyake, T.}
\newblock {\em Modular {Forms}}.
\newblock Springer {Monographs} in {Math}. Springer, 1989.

\bibitem{montgomery_multiplicative_2006}
{\sc Montgomery, H.~L., and Vaughan, R.~C.}
\newblock {\em Multiplicative {Number} {Theory} {I}: {Classical} {Theory}}.
\newblock Cambridge University Press, 2006.

\bibitem{petrow_weyl_2018}
{\sc Petrow, I.}
\newblock The {Weyl} law on algebraic tori.
\newblock arXiv: 1808.09991.

\bibitem{risager_automorphic_2003}
{\sc Risager, M.~S.}
\newblock {\em Automorphic forms and modular symbols}.
\newblock PhD thesis, University of Aarhus, 2003.

\bibitem{sarnak_definition_2008}
{\sc Sarnak, P.}
\newblock Definition of {Families} of {L}-functions, 2008.
\newblock Available at publications.ias.edu/sarnak.

\bibitem{selberg_harmonic_1989}
{\sc Selberg, A.}
\newblock Harmonic analysis.
\newblock In {\em Coll. {Papers}}, vol.~1. Springer-Verlag, 1989, pp.~626--674.

\end{thebibliography}

\end{document}